%% file: ex_article.tex
\documentclass[hidelinks,onefignum,onetabnum]{siamart250211}

\input{ex_shared}
\ifpdf
\hypersetup{
  pdftitle={importance sampling and active subspace in quasi-Monte Carlo},
  pdfauthor={J. Yu, X. Wang}
}
\fi

\usepackage{multirow}
\begin{document}
\setlength{\oddsidemargin}{\dimexpr(\paperwidth-\textwidth)/2-1in-\hoffset\relax}
\setlength{\evensidemargin}{\oddsidemargin}
\maketitle

\begin{abstract}
Quasi-Monte Carlo (QMC) is widely used in computational finance, and its efficiency strongly depends on the smoothness and effective dimension of the integrand.
In this work, we propose a three-step approach, referred to as the IS\_AS\_PI method, which sequentially applies importance sampling (IS), the active subspace (AS) method and preintegration (PI).
The proposed method overcomes the problem of empirical failure in constructing active subspaces when nonzero payoff gradients are rarely sampled.
Here, empirical failure means that the estimated gradient information matrix is zero and thus provides no useful directions.
We establish the orthogonal invariance of active subspaces and IS-active subspaces under suitable conditions.
The method is applied to pricing Asian and basket options and estimating the Delta of Asian options under the Black--Scholes model.
Numerical experiments show that the proposed method avoids this empirical failure in the deep out-of-the-money cases.
Moreover, in all experiments, the proposed method with Laplace IS achieves the largest variance reduction factors, indicating the best performance, among the methods compared.
\end{abstract}



\section{Introduction} 
Many problems in computational finance can be formulated as high-\-dim\-en\-sion\-al integrals over \(\mathbb{R}^d\), where the positive integer \(d\) is the problem dimension. 
The integral of interest is
\begin{equation}
  I(g) = \int_{\mathbb{R}^d} g(\bm{z})\, \mathrm{d}\bm{z}. 
  \label{eq:c1integration}
\end{equation}
Such integrals typically lack explicit solutions and require numerical approximation.
Traditional numerical methods may not be used due to the curse of dimensionality, which means the number of nodes grows exponentially with increasing dimension.  

Monte Carlo (MC) methods provide a feasible solution.  
MC methods converge at a rate of \(O(n^{-1/2})\) in the probability sense \cite{glassermanMonteCarloMethods2003}, where \(n\) is the number of sampling points.
Although the convergence rate is independent of the dimension, the low convergence rate  motivates the development of quasi-Monte Carlo (QMC) methods, where deterministic sample points are used rather than random points.
QMC methods achieve a convergence rate of \(O\left(n^{-1}(\log n)^d\right)\) \cite{niederreiterRandomNumberGeneration1992}, which is higher than that of MC methods in the asymptotic sense. 

QMC was not recommended for high-dimensional problems because its convergence rate depends on the dimension \(d\), until Paskov and Traub \cite{paskovFasterValuationFinancial1995} applied it to solve the high-dimensional problem of pricing mortgage-backed securities. 
The numerical results demonstrated that QMC methods significantly outperformed MC methods for this high-dimensional problem.
Subsequent studies have reported similar observations for a wide range of high-dimensional problems in computational finance \cite{akessonPathGenerationQuasiMonte2000,wangQuasiMonteCarloMethods2011}, indicating that the superior performance of QMC is not a problem-specific anomaly. 

These findings immediately attracted considerable attention in the academic community and led to the development of various theoretical explanations, among which the theory of effective dimension is the most representative \cite{caflischValuationMortgagebackedSecurities1997,wangEffectiveDimensionQuasiMonte2003}. According to the theory of effective dimension, the performance of QMC depends on the effective dimension of the problem rather than its nominal dimension. 
Suitable changes of variables preserve the integral value
but can produce a transformed integrand with a lower
effective dimension, potentially improving QMC efficiency.

There are two kinds of methods to reduce the effective dimension. Early efforts primarily focused on path generation methods, such as the Brownian bridge (BB) method \cite{caflischValuationMortgagebackedSecurities1997} and the principal component analysis (PCA) method \cite{acworthComparisonMonteCarlo1998a}. 
Many studies have investigated the impact of path generation methods on the effective dimension \cite{
  liuEstimatingMeanDimensionality2006,
  wangEffectsDimensionReduction2006,wangEffectiveDimensionQuasiMonte2003,wangHowPathGeneration2012}. 
In particular, Papageorgiou showed that the BB method can have a negative effect in digital option pricing problems \cite{papageorgiouBrownianBridgeDoes2002}. 
The reason is that BB does not consider the characteristics of the integrand, and its performance therefore varies with the specific form of the integrand.

Recently, scholars have paid more attention to dimension reduction methods that consider the features of the integrand.
Imai and Tan proposed the linear transformation (LT) method \cite{imaiGeneralDimensionReduction2006}, which approximates the integrand by a first-order Taylor expansion at some fixed points.
Weng et al.\ \cite{wengEfficientComputationOption2017} used the \(k\)-means algorithm and QR decomposition to find a \(k\)-dimensional  approximation of the integrand. 
Xiao and Wang used the gradient information from the integrand and proposed the gradient PCA (GPCA) method \cite{xiaoConditionalQuasiMonteCarlo2018,xiaoEnhancingQuasiMonteCarlo2019}. 

Importance sampling is a classical variance reduction technique in MC methods. Reider applied importance sampling to the pricing of deep out-of-the-money options in 1993 \cite{Reiderefficientmontecarlo1993}, and the method was subsequently applied more broadly to problems in computational finance \cite{boyleMonteCarloMethods1997}. More recently, Zhang et al.\ combined importance sampling with the GPCA method to further reduce variance \cite{zhangEfficientImportanceSampling2021}.

The active subspace method \cite{constantineActiveSubspacesEmerging2015} is a dimension-reduction technique that has emerged in recent years. It has been successfully applied to option pricing \cite{liuPreintegrationActiveSubspace2023}, satellite system design \cite{huDiscoveringOnedimensionalActive2016}, hydrological modeling \cite{jeffersonActiveSubspacesSensitivity2015}, and other problems. 
Liu and Owen applied the active subspace method to option pricing problems in computational finance \cite{liuPreintegrationActiveSubspace2023}, achieving superior performance over GPCA in certain cases.

In this paper, we combine importance sampling with the active subspace method and propose a three-step method. 
The proposed method overcomes the problem of empirical failure in constructing active subspaces for deep out-of-the-money option pricing problems when nonzero payoff gradients are rarely sampled.
The proposed method is applied to option pricing and sensitivity analysis problems, with a focus on pricing Asian and basket options and estimating the Delta of Asian options under the Black--Scholes model, and is compared with several existing approaches.
Moreover, we establish the orthogonal invariance of active subspaces and IS-active subspaces under the stated assumptions with the corresponding choice of eigenvectors, providing a basis for using the standard construction of Brownian motion in our numerical experiments. Through theoretical analysis and numerical experiments, we assess the advantages and limitations of the proposed method.

The rest of the paper is organized as follows. 
The necessary background knowledge is reviewed briefly in Section \ref{sec:bg}, including the variance reduction techniques involved in our three-step method. 
Section \ref{sec:ISAS}  presents our three-step method, as well as the  orthogonal invariance of active subspaces and importance sampling-active subspaces. 
In Section~\ref{sec:4Application}, we apply the proposed
method to option pricing and sensitivity analysis.
Numerical results are reported in Section~\ref{sec:Experiment}.
The conclusions follow in Section \ref{sec:conclusion}.

\section{Background} 
\label{sec:bg}

In this section, we present some background on preintegration, active subspace and importance sampling. 
Let $g(\bm{z})$ be a \(d\)-dimensional function and let $\bm{Z}\sim N(\bm 0,I_d)$. We consider the integration problem
\begin{equation}
  \mathbb{E}\left(g(\bm{Z})\right)
  .\label{eq:c3expaction}
\end{equation}

First, we introduce some notation.
For a \(d\)-dimensional vector \(\bm z\), we define
\(\bm z_{-1}:=(z_2,\dots,z_d)^{\mathsf{T}}\),
which is a \((d-1)\)-dimensional vector consisting of
the last \(d-1\) components of \(\bm z\).
For \(\bm Z\sim N(\bm 0,I_d)\), \(\bm Z_{-1}\) similarly denotes
the vector consisting of the last \(d-1\) components of \(\bm Z\),
and hence is a \((d-1)\)-dimensional random vector with distribution \(N(0,I_{d-1})\).
The density and cumulative distribution function (CDF) of the standard Gaussian distribution \(\mathcal{N}(0,1)\) are denoted by \(\varphi\) and \(\Phi\), respectively. 
These symbols are also used for the multi-dimensional standard Gaussian distribution. 
For a positive integer \(d\), we let \(\{1:d\} =\{1,2,\dots,d\}\). 
For a matrix \(C\), \(C_{\left[i,j\right]}\) is used to denote the \((i,j)\)-element of \(C\), and \(C_{\left[i,:\right]}\) and \(C_{\left[:,j\right]}\) are used to denote the \(i\)-th row of \(C\) and the \(j\)-th column of \(C\), respectively. 

\subsection{Preintegration}


Preintegration (PI) is a method based on the law of total expectation.
It first evaluates the inner conditional expectation by integrating out a subset of the variables, either in closed form or using a sufficiently accurate quadrature rule, and then estimates the outer expectation using MC or QMC.

\subsubsection{Basic Principle}

Typically, just one variable is preintegrated. Without loss of generality, we suppose that the first variable is preintegrated. Then \eqref{eq:c3expaction} can be rewritten as
\begin{equation}
  \mathbb{E}\left(g(\bm{Z})\right)=\mathbb{E}\left(\mathbb{E}\left(g(\bm{Z})|\bm{Z}_{-1}\right)\right).
  \label{eq:c3doubleexpaction}
\end{equation}
We define the inner conditional expectation as
\begin{equation}
  g^\mathrm{P}(\bm{z}_{-1}):=\mathbb{E}\left(g(\bm{Z})|\bm{Z}_{-1}=\bm{z}_{-1}\right),\label{eq:c3doubleexpactioninner}
\end{equation}
which depends on the last $d-1$ variables $\bm{z}_{-1}$.  
Then the original integral $\mathbb{E}\left(g(\bm{Z})\right)$ is calculated by estimating the outer expectation $\mathbb{E}\left(g^\mathrm{P}(\bm{Z}_{-1})\right)$, where \(\bm{Z}_{-1}\sim N(0,I_{d-1})\).

In MC methods, preintegration is also called
conditional Monte Carlo (CMC), a variance reduction
technique \cite{glassermanMonteCarloMethods2003}.
Moreover, preintegration can improve the smoothness of the integrand, which is of great significance in QMC methods.

\subsubsection{Variance Reduced by Preintegration}
Let the ANOVA decomposition of $g(\bm{z})$ be $g(\bm{z})=\sum_{u\subseteq \{1:d\}} g_u(\bm{z}_u)$, where \(g_u(\bm{z}_u)\) only depends on the components of \(\bm{z}\) indexed by \(u\).  
The expression of $g^\mathrm{P}(\bm{z}_{-1})$ defined in \eqref{eq:c3doubleexpactioninner} can be written as
\begin{align}
  g^\mathrm{P}(\bm{z}_{-1})
  =\sum_{u\subseteq \{1:d\}} \int_{\mathbb{R}} g_u(\bm{z}_u)\varphi(z_1) \mathrm{d} z_1
  =\sum_{u\subseteq \{2:d\}} g_u(\bm{z}_u),
\end{align}
by using the property of the ANOVA decomposition that \(\int_{-\infty}^{\infty}g_{u}(\bm{z}_u)\varphi(z_j)\mathrm{d}z_j=0, \forall j\in  u\), where \(\varphi\) is the density of \(N(0,1)\).
Thus the reduction in variance obtained by integrating out the first variable is
\begin{align}
  \sigma^2(g)-\sigma^2(g^\mathrm{P})
  =\sum_{u\subseteq \{1:d\}}\sigma^2(g_u)-\sum_{u\subseteq \{2:d\}}\sigma^2(g_u)
  =\sum_{u\cap \{1\}\neq \emptyset}\sigma^2(g_u)
  =\overline{\tau}_{\{1\}}^2(g),\notag
\end{align}
where \(\overline{\tau}_{\{1\}}^2(g)\) is the Sobol' upper index associated with the first variable of \(g(\bm{z})\) \cite{sobolSensitivityEstimatesNonlinear1993}.

Obviously, the more variance is contributed by the preintegrated variable, the more variance will be reduced through preintegration.
Therefore, we aim to preintegrate the variable with the largest Sobol' upper index, in order to maximize the reduction in integrand variance. 
However, smoothness is also a vital factor affecting the efficiency of QMC methods, so both variance reduced and smoothness should be taken into consideration when selecting the preintegration variable in QMC methods. 

\subsubsection{Smoothing Effect of Preintegration}
\label{sec:ch2SmoothingEffect}
Wang \cite{wangHandlingDiscontinuitiesFinancial2016}, He and Wang \cite{heDimensionReductionSmoothing2017} proposed the variable separation method, which smooths integrands with variable-separable discontinuities to enhance QMC performance. 
Let \(k_1(\bm{z}),k_2(\bm{z})\) be continuous functions, and suppose that the integrand in \eqref{eq:c3expaction} has the form
\begin{equation}
  g(\bm{z})=k_1(\bm{z})\bm{1}\{k_2(\bm{z})>0\},\label{eq:c3intergrand1}
\end{equation}
where $\bm{1}\{k_2(\bm{z})>0\}$ denotes the indicator function of the set $\{k_2(\bm{z})>0\}$.  
Consider this discontinuous factor $\bm{1}\{k_2(\bm{z})>0\}$ in \eqref{eq:c3intergrand1}. We assume that for some fixed \(\bm z_{-1}\), \(k_2(z_1,\bm z_{-1})=0\) has at least one finite root. Otherwise, the indicator is constant and no finite integration boundary is needed. 
For illustration, we consider the following two cases.
The assumptions in these cases are not required for
preintegration in general.

In the first case, assume that, for each fixed \(\bm z_{-1}\), \(k_2(z_1,\bm z_{-1})\) is strictly monotone in \(z_1\) and \(k_2(z_1,\bm z_{-1})=0\) has a unique finite root \(\gamma=\gamma(\bm{z}_{-1})\) which is continuous with respect to \(\bm z_{-1}\). Then, in the increasing case, we have 
\begin{equation}
  \bm{1}\{k_2(\bm{z})>0\}=\bm{1}\{z_1>\gamma\},
  \label{eq:c3variableseperate1}
\end{equation}  
and in the decreasing case, we have
\begin{equation}
  \bm{1}\{k_2(\bm{z})>0\}=\bm{1}\{z_1<\gamma\}.
  \label{eq:c3variableseperate2}
\end{equation}
For example, when the equation~\eqref{eq:c3variableseperate1} holds, the preintegration step yields
\[
g^\mathrm{P}(\bm{z}_{-1})
=
\int_{\gamma(\bm{z}_{-1})}^{\infty}
k_1(t,\bm z_{-1})\varphi(t)\,\mathrm dt,
\]
where \(\varphi\) is the density of \(N(0,1)\).
The indicator is removed from the integrand \(g(\bm{z})\).
Under suitable conditions, the smoothness of the resulting function \(g^\mathrm{P}(\bm{z}_{-1})\) is improved
\cite{griewankHighDimensionalIntegration2018,heErrorRateConditional2019}.

In the second case, suppose that, for each fixed \(\bm{z}_{-1}\),
\(k_2(z_1,\bm z_{-1})=0\) has exactly two finite roots
\(\gamma_1(\bm{z}_{-1})<\gamma_2(\bm{z}_{-1})\),
which are continuous functions of \(\bm z_{-1}\),
and $k_2(\bm{z})$ changes sign at both roots.
Preintegration can then be performed over
\((\gamma_1,\gamma_2)\) or
\((-\infty,\gamma_1)\cup(\gamma_2,\infty)\),
depending on where $k_2(\bm{z})$ is positive.

More generally, the number of roots of
\(k_2(z_1,\bm z_{-1})=0\) may depend on \(\bm z_{-1}\).
For each fixed \(\bm z_{-1}\), preintegration is performed
over the set where \(k_2(z_1,\bm z_{-1})>0\).
When a closed-form expression for the preintegrated
function is unavailable, 
a high-precision quadrature rule can be used in the preintegration step 
\cite{liuPreintegrationActiveSubspace2023}.

Therefore, preintegration can both reduce variance and improve smoothness, making it particularly compatible with QMC methods.
Liu and Owen developed a method providing an effective direction for preintegration by active subspaces \cite{liuPreintegrationActiveSubspace2023}.

\subsection{Active Subspace Method}

Let $g(\bm{z})$ in the integration problem  \eqref{eq:c3expaction} be a function that is differentiable almost everywhere, with a square-integrable gradient under the standard Gaussian measure.
The active subspace (AS) method \cite{constantineActiveSubspacesEmerging2015} is a dimension reduction technique. It ranks orthogonal directions in descending order of their mean-squared directional derivatives.

\subsubsection{Active Subspace}

Consider the matrix
\begin{equation}
  C=\mathbb{E}\left(\nabla g(\bm{Z})\nabla g(\bm{Z})^{\mathsf{T}} \right),\label{eq:c3gradientC}
\end{equation}
where \(\bm Z\sim N(\bm 0,I_d)\) and 
\(\nabla g=\left(\frac{\partial g}{\partial z_1},\frac{\partial g}{\partial z_2},\dots,\frac{\partial g}{\partial z_d}\right)^{\mathsf{T}}.\)
By definition, $C$ contains information about the gradients of $g$ and is referred to as the gradient information matrix.
Obviously, $C$ is positive semidefinite and admits the eigendecomposition
\begin{equation}
  C=Q\Lambda Q^{\mathsf{T}},\label{eq:c3eigenC}
\end{equation}
where $Q=[\bm{q}_1,\bm{q}_2,\dots,\bm{q}_d]$ is an orthogonal matrix of eigenvectors of \(C\), and the diagonal matrix $\Lambda = \mathrm{diag}(\lambda_1,\lambda_2,\cdots,\lambda_d)$ contains eigenvalues satisfying \(\lambda_1\ge\cdots\ge\lambda_d\ge0\), with \(\lambda_i=\mathbb E[(\bm q_i^{\mathsf{T}}\nabla g(\bm Z))^2]\), \(1\le i\le d\). 

Let $Q_{[:,1:r]}=[\bm{q}_1,\bm{q}_2,\dots,\bm{q}_r]$ denote the first $r$ columns of $Q$.  
The $r$-dimensional space spanned by the columns of $Q_{[:,1:r]}$ is called the $r$-dimensional active subspace of $g(\bm{z})$.  
In this paper, active subspaces are defined under the standard Gaussian density; that is, the expectation defining \(C\) is taken with respect to the standard Gaussian density. 
Let 
\begin{equation}
  (y_1,y_2,\dots,y_r)^{\mathsf{T}}
  =\left(Q_{[:,1:r]}\right)^{\mathsf{T}}
  (z_1, z_2, \dots, z_d)^{\mathsf{T}}, 
  \label{eq:c3ASCoordinateTransform}
\end{equation}
where $y_1,y_2,\dots,y_r$ are called the active variables of $g(\bm{z})$.
When $r=d$, \eqref{eq:c3ASCoordinateTransform} is simply a variable transformation $\bm{y}=Q^{\mathsf{T}}\bm{z}$, so that $g(\bm{z})=g(Q\bm{y})$.
We define 
\begin{equation}
  g^\mathrm{A}(\bm{y}):=g(Q\bm{y}).
\end{equation}

\subsubsection{Relations to the Sobol' Upper Index}

The original motivation for introducing the AS method was to approximate a $d$-dimensional function using only its first $r$ active variables \cite{constantineActiveSubspaceMethods2014}.  
When the full $d$-dimensional active subspace is used, the transformed variables are ranked in descending order of their mean-squared partial derivatives.

For $g^\mathrm{A}(\bm{y}) = g(\bm{q}_1 y_1 + \dots + \bm{q}_d y_d)$, the chain rule gives \(\frac{\partial g^\mathrm{A}}{\partial y_i}(\bm{y})=\bm{q}_i^{\mathsf{T}}\nabla g(Q\bm{y})\) almost everywhere. We have 
\begin{equation}
  \begin{aligned}
    \mathbb E\left(\left(\frac{\partial g^\mathrm{A}}{\partial y_i}(\bm{Z})\right)^2\right)
=\mathbb E\left(\bm{q}_i^{\mathsf{T}}\nabla g(Q\bm{Z})\nabla g(Q\bm{Z})^{\mathsf{T}} \bm{q}_i\right) 
= \bm{q}_i^{\mathsf{T}} C \bm{q}_i
=\lambda_i,
  \end{aligned}
\end{equation}
where the second equality follows from \(Q\bm Z\sim N(\bm 0,I_d)\).
Thus the coordinates of \(g^\mathrm{A}\) are ranked in descending order of their mean-squared partial derivatives.

Let $\overline{\tau}_{\{1\}}^2(g^\mathrm{A})$ denote the Sobol' upper index associated with the first variable $y_1$ of the transformed function  $g^\mathrm{A}(\bm{y})$.  
According to the Poincaré inequality \cite{sobolDerivativeBasedGlobal2009}, the Sobol' upper index has an upper bound:
\begin{align}
\overline{\tau}_{\{1\}}^2(g^\mathrm{A})
\leq \mathbb{E}\left(\left(\bm{q}_1^{\mathsf{T}}\nabla g (\bm Z) \right)^2\right) = \lambda_1.
\end{align}
For any unit vector $\bm{q}$, $\bm{q}^{\mathsf{T}}C\bm{q}$ is bounded above by the largest eigenvalue of $C$. Within the active subspace framework, the upper bound for the Sobol' upper index of the first active variable attains its maximum. 
Thus the AS criterion provides a computationally convenient way to select a preintegration direction, although it does not guarantee the largest possible reduction in integrand variance.

\subsubsection{Active Subspace and Preintegration}
\label{sec:223ASandPreint}
Preintegration aims to integrate out a variable with a large Sobol’ upper index.  
The AS method provides a gradient-based criterion for choosing the preintegration direction.  
Recent studies have combined preintegration with the AS method and reported exceptional performance on various problems in computational finance \cite{liuPreintegrationActiveSubspace2023}.
The eigendecomposition of the gradient information matrix \(C\) corresponds to the uncentered principal component analysis of gradients, also known as GPCA \cite{liuPreintegrationActiveSubspace2023}. 
Xiao and Wang proposed a method that performs
preintegration first and then applies GPCA
\cite{xiaoConditionalQuasiMonteCarlo2018}.
We refer to this method as PI\_GPCA. The method is summarized in
Algorithm~\ref{al:preintGPCA}.
In this algorithm, GPCA constructs an orthogonal transformation using the eigenvectors of the gradient information matrix of the 
\(\left(d-1\right)\)-dimensional preintegrated integrand.
Liu and Owen proposed a reverse strategy called AS\_PI, which applies the AS method before preintegration \cite{liuPreintegrationActiveSubspace2023}. 
The method is presented in Algorithm~\ref{al:ASpreint}.  

\begin{algorithm}[!t]
  \caption{PI\_GPCA}
  \label{al:preintGPCA}
  \small
  \begin{algorithmic}
    \STATE (1) Compute the preintegrated function $g^\mathrm{P}(\bm{z}_{-1})$ according to \eqref{eq:c3doubleexpactioninner};
    \STATE (2) Compute the gradient $\nabla g^\mathrm{P}(\bm{z}_{-1})$ of the $(d-1)$-dimensional function $g^\mathrm{P}(\bm{z}_{-1})$;
    \STATE (3) Using QMC points $\{\bm{z}_{-1}^{(1)},\dots,\bm{z}_{-1}^{(M)}\}$ from $N(0,I_{d-1})$, estimate the gradient information matrix by $C \approx  \frac{1}{M}\sum_{k=1}^{M} \nabla g^\mathrm{P}(\bm{z}_{-1}^{(k)})\nabla g^\mathrm{P}(\bm{z}_{-1}^{(k)})^{\mathsf{T}}$;
    \STATE (4) Perform the eigendecomposition
     $C = Q\Lambda Q^{\mathsf{T}}$, and define \(g^{\mathrm{PA}}(\bm{z}_{-1})=g^\mathrm{P}(Q\bm{z}_{-1})\);
    \STATE (5) Evaluate $\mathbb{E}\!\left(g^{\mathrm{PA}}(\bm{Z}_{-1})\right)$ by QMC, where \(\bm{Z}_{-1}\sim N(\bm 0,I_{d-1})\).
  \end{algorithmic}
\end{algorithm}


\begin{algorithm}[!t]
  \caption{AS\_PI}
  \label{al:ASpreint}
  \small
  \begin{algorithmic}
    \STATE (1) Compute the gradient $\nabla g(\bm{z})$;
    \STATE (2) 
    Using QMC points $\{\bm{z}^{(1)},\dots,\bm{z}^{(M)}\}$ from $N(0,I_{d})$, estimate the gradient information matrix by
      $C \approx  \frac{1}{M}\sum_{k=1}^{M} \nabla g(\bm{z}^{(k)})\nabla g(\bm{z}^{(k)})^{\mathsf{T}}$;
    \STATE (3) Perform the eigendecomposition $C = Q\Lambda Q^{\mathsf{T}}$ and define $g^\mathrm{A}(\bm{z}) = g(Q\bm{z})$;
    \STATE (4) Compute the preintegrated function defined by 
    $g^\mathrm{AP}(\bm{z}_{-1}) = \mathbb{E}\!\left(g^\mathrm{A}(\bm{Z})\mid \bm{Z}_{-1}=\bm{z}_{-1}\right)$;
    \STATE (5) Evaluate $\mathbb{E}\!\left(g^\mathrm{AP}(\bm{Z}_{-1})\right)$ by QMC, where \(\bm{Z}_{-1}\sim N(\bm 0,I_{d-1})\).
  \end{algorithmic}
\end{algorithm}


Both PI\_GPCA and AS\_PI algorithms have demonstrated substantial reductions in effective dimension and strong numerical performance in option pricing and sensitivity analysis.  
However, directly constructing active subspaces in AS\_PI may lead
to empirical failure for some deep out-of-the-money options.
For such options, samples drawn from the standard
normal distribution have a high probability of falling
into the region where the gradient of the payoff is zero.
More precisely, all the samples used to estimate $C$ may result in
zero gradients.
In this case, the estimated $C$ is a zero matrix and provides
no useful information for preintegration.
We refer to this situation as empirical failure. 
The empirical failure frequency is evaluated by numerical experiments in Sections~\ref{sec:AsianFailure} and~\ref{sec:BasketFailure}.

\subsection{Importance Sampling}

Let $g(\bm{z})$ in the integration problem \eqref{eq:c3expaction} be a square integrable function.
Importance sampling (IS) is a powerful variance reduction technique that modifies the sampling distribution to place greater weight on important regions of the integration domain.

\subsubsection{Basic Principle}

Let $h(\bm{z})$ be a density whose support contains that of $\varphi(\bm{z})$.  
Then
\begin{equation}
  \mathbb{E}\left(g(\bm{Z})\right)
  = \int g(\bm{z})\varphi(\bm{z})\,\mathrm{d}\bm{z} 
  = \int g(\bm{z}) \frac{\varphi(\bm{z})}{h(\bm{z})} h(\bm{z})\,\mathrm{d}\bm{z} 
  = \mathbb{E}_{h}\!\left(g(\bm{Z})\frac{\varphi(\bm{Z})}{h(\bm{Z})}\right),
  \label{eq:c3ISprinciple}
\end{equation}
where \(\mathbb{E}_{h}\!\left(\cdot\right)\) denotes the expectation with respect to the density \(h(\bm{z})\). The density \(h(\bm{z})\) is called the IS density. 
The variance under the IS density is 
\begin{equation}
  \int \left(g(\bm{z})\omega(\bm{z})-\alpha\right)^2 h(\bm{z})\,\mathrm{d}\bm{z},
  \label{eq:c3ISvarriance}
\end{equation}
where $\omega(\bm{z})=\varphi(\bm{z})/h(\bm{z})$ and $\alpha=\mathbb{E}_{h}(g\omega)$.  
To minimize the variance, $h$ should ideally approximate the optimal choice.
If $g(\bm{z}) \ge 0$, the optimal density is
\begin{equation}
  h_{\mathrm{opt}}(\bm{z})
  = \frac{g(\bm{z})\varphi(\bm{z})}{\alpha}.
  \label{eq:c3ISdensityoptimal}
\end{equation}
Although \eqref{eq:c3ISdensityoptimal} yields zero variance, it depends on the unknown integral $\alpha$ and is therefore not implementable.  
Nevertheless, it motivates selecting $h$ so that
\[
  h(\bm{z}) \propto g(\bm{z})\varphi(\bm{z}).
\]


In this work, we consider optimal drift IS
\cite{glassermanAsymptoticallyOptimalImportance1999,suOptimalImportanceSampling2002}
and Laplace IS
\cite{kukLaplaceImportanceSampling1999,kuoQuasiMonteCarloHighly2008},
both of which use a Gaussian IS density \(h\).
Gaussian samples from \(h\) are generated from uniform
QMC points using the inverse standard normal CDF followed
by an affine transformation.
QMC is not restricted to Gaussian densities. More general
distributions can be handled using the transport method in
\cite{liuTransportQuasiMonteCarlo2026}.

\subsubsection{Optimal Drift Importance Sampling}

If $h$ is constrained to the family $N(\boldsymbol{\mu},I_d)$, then only the drift $\boldsymbol{\mu}$ needs to be selected. 
Such IS schemes are referred to as optimal drift IS.
In this case,
\begin{align}
  \mathbb{E}\left(g(\bm{Z})\right)
  &=\mathbb{E}_{\boldsymbol{\mu}}\!\left(
    g(\bm{Z})
    \exp\!\left(-\boldsymbol{\mu}^{\mathsf{T}}\bm{Z}
    +\tfrac{1}{2}\boldsymbol{\mu}^{\mathsf{T}}\boldsymbol{\mu}\right)\right)  \notag\\
  &=\mathbb{E}\!\left(
    g(\bm{Z}+\boldsymbol{\mu})
    \exp\!\left(-\boldsymbol{\mu}^{\mathsf{T}}\bm{Z}
    -\tfrac{1}{2}\boldsymbol{\mu}^{\mathsf{T}}\boldsymbol{\mu}\right)\right),
  \label{eq:c3ISODISpriciple}
\end{align}
where $\mathbb{E}_{\bm{\mu}}$ denotes the expectation under the measure induced by $N(\bm{\mu},I_d)$, and the second equality follows from an affine change of variables.
Assume that \(g\) is differentiable on \(D_+=\{\bm{z}:g(\bm{z})>0\}\). Following Glasserman et al.\
\cite{glassermanAsymptoticallyOptimalImportance1999},
we select the drift vector $\boldsymbol{\mu}^*$ by solving
\begin{equation}
  \nabla F(\bm{z}) = \bm{z},
  \label{eq:c3ISODISmiu0}
\end{equation}
where \(F(\bm{z})\) is defined on \(D_+\) by 
\begin{equation}
  F(\bm{z}) = \ln g(\bm{z}).
  \label{eq:c3ISFlogg}
\end{equation}
Equation~\eqref{eq:c3ISODISmiu0} is the 
first-order condition for maximizing
\(g(\bm z)\varphi(\bm z)\) on \(D_+\).
This choice is motivated by approximating the ideal
IS density with a Gaussian centered at a maximizer
of $g(\bm z)\varphi(\bm z)$, 
with the aim of minimizing
the variance of the IS estimator within the family
\(N(\boldsymbol{\mu},I_d)\).
The equation~\eqref{eq:c3ISODISmiu0} can equivalently be written as
\begin{equation}
  \frac{\nabla g(\bm{z})}{g(\bm{z})} = \bm{z}.
  \label{eq:c3ISODISmiu}
\end{equation}
Define the transformed integrand as
\begin{equation}
  g^\mathrm{I}(\bm{z})
  := g(\bm{z}+\boldsymbol{\mu}^*)
    \exp\!\left(
      -{\boldsymbol{\mu}^*}^{\mathsf{T}}\bm{z}
      -\tfrac{1}{2}{\boldsymbol{\mu}^{*}}^{\mathsf{T}}\boldsymbol{\mu}^*
    \right).
  \label{eq:c3ISODISintegrand}
\end{equation}
Then \(\mathbb{E}(g(\bm{Z})) = \mathbb{E}(g^\mathrm{I}(\bm{Z}))\).

\subsubsection{Laplace Importance Sampling}
\label{sec:lapIS}
More generally, consider selecting the IS density from the family of normal distributions $N(\bm{\mu},\Gamma)$, where $\Gamma$ is a positive definite matrix. In this case, one needs to determine the drift vector $\bm{\mu}$ and the covariance matrix $\Gamma$. Such IS schemes are referred to as Laplace IS.
In this case, equation~\eqref{eq:c3ISprinciple} can be written as
\begin{align}
\mathbb{E}\!\left(g(\bm{Z})\right)
&= \mathbb{E}_{\bm{\mu},\Gamma} \!\left(
g(\bm{Z})\sqrt{\det(\Gamma)}\,
\exp\!\left(-\tfrac{1}{2}\bm{Z}^{\mathsf{T}}\bm{Z}
+\tfrac{1}{2}(\bm{Z}-\bm{\mu})^{\mathsf{T}}\Gamma^{-1}(\bm{Z}-\bm{\mu})\right)
\right) \notag\\
&=\mathbb{E}\!\left(
g(L\bm{Z}+\bm{\mu})\sqrt{\det(\Gamma)}\,
\exp\!\left(\tfrac{1}{2}\bm{Z}^{\mathsf{T}}\bm{Z}
-\tfrac{1}{2}(L\bm{Z}+\bm{\mu})^{\mathsf{T}}(L\bm{Z}+\bm{\mu})\right)
\right),
\label{eq:c3ISLapISpriciple}
\end{align}
where $\mathbb{E}_{\bm{\mu},\Gamma}$ denotes the expectation under the measure induced by $N(\bm{\mu},\Gamma)$, the matrix $L$ satisfies $LL^{\mathsf{T}}=\Gamma$, and the last step follows by applying an affine change of variables.

Using the same notation as in~\eqref{eq:c3ISFlogg}, the selection of the density in Laplace IS also requires the optimal drift $\bm{\mu}^*$ to satisfy equation~\eqref{eq:c3ISODISmiu0}. 
Assume that \(F\) is twice continuously differentiable in a neighborhood of the selected drift \(\bm{\mu}^*\).
Once $\bm{\mu}^*$ is obtained, the covariance matrix is chosen as \cite{boothMaximizingGeneralizedLinear1999,kukLaplaceImportanceSampling1999,kuoQuasiMonteCarloHighly2008} 
\begin{equation}
\Gamma^*=\left(I_d-\nabla^2 F\!\left(\bm{\mu}^*\right)\right)^{-1},
\label{eq:c3ISLapISGamma}
\end{equation}
where $\nabla^2 F(\bm{z})$ denotes the Hessian matrix of $F(\bm{z})$.
The matrix \(\Gamma^*\) is assumed to be positive definite.
Define the integrand as 
\begin{equation}
g^\mathrm{I}(\bm{z}):=
g(L^*\bm{z}+\bm{\mu}^*)\,
\sqrt{\det(\Gamma^*)}\,
\exp\!\left(
\tfrac{1}{2}\bm{z}^{\mathsf{T}}\bm{z}
-\tfrac{1}{2}(L^*\bm{z}+\bm{\mu}^*)^{\mathsf{T}}(L^*\bm{z}+\bm{\mu}^*)
\right),
\label{eq:c3ISLapISintegrand}
\end{equation}
where $L^*{L^*}^{\mathsf{T}}=\Gamma^*$. Then \(\mathbb{E}(g(\bm{Z})) = \mathbb{E}(g^\mathrm{I}(\bm{Z}))\).

Here, optimal drift IS and Laplace IS are developed for nonnegative integrands.
For a signed integrand, one possible approach is to
decompose it as 
\begin{equation}
  g=g^+-g^-,
  \label{c3:SignedIS}
\end{equation}
where \(g^+(\bm{z})=\max\{g(\bm{z}),0\}\) and \(g^-(\bm{z})=\max\{-g(\bm{z}),0\}\).
Using the identity
\[
\mathbb{E}[g(\bm{Z})]
=
\mathbb{E}[g^+(\bm{Z})]
-
\mathbb{E}[g^-(\bm{Z})],
\]
we can estimate \(\mathbb{E}[g(\bm{Z})]\) by applying
IS separately to the two nonnegative parts and
subtracting the resulting estimates, as discussed
by Zhang et al.\
\cite{zhangEfficientImportanceSampling2021}.

\subsubsection{Importance Sampling and PI\_GPCA}

Zhang et al.\ integrated IS into the PI\_GPCA method and proposed a method that performs IS after preintegration, and subsequently applies GPCA \cite{zhangEfficientImportanceSampling2021}. We refer to this method as the PI\_IS\_GPCA method. 
The algorithmic procedure is summarized in Algorithm~\ref{al:preintISGPCA}. 
This method achieves faster convergence than the PI\_GPCA method in certain scenarios.

\begin{algorithm}[!tb]
\caption{PI\_IS\_GPCA}
\label{al:preintISGPCA}
\small
\begin{algorithmic}
\STATE (1) Compute the preintegrated function $g^\mathrm{P}(\bm{z}_{-1})$ according to~\eqref{eq:c3doubleexpactioninner};
\STATE (2) 
For the $(d-1)$-dimensional function $g^\mathrm{P}(\bm{z}_{-1})$, compute the drift $\bm{\mu}^*$ using~\eqref{eq:c3ISODISmiu};
\STATE\hspace*{1em} -- If optimal drift IS is used, define $g^\mathrm{PI}(\bm{z}_{-1})$ using \eqref{eq:c3ISODISintegrand};
\STATE\hspace*{1em} -- If Laplace IS is used, compute $\Gamma^*$ using~\eqref{eq:c3ISLapISGamma} and its Cholesky factor $L^*$; define $g^\mathrm{PI}(\bm{z}_{-1})$ using \eqref{eq:c3ISLapISintegrand};
\STATE\hspace*{1em} In these formulas, replace $g$ by $g^\mathrm{P}$ and use dimension $d-1$;
\STATE (3) 
Compute the gradient vector $\nabla g^\mathrm{PI}$;
\STATE (4) 
Using QMC points
$\{\bm{z}_{-1}^{(1)},\dots,\bm{z}_{-1}^{(M)}\}$ from $N(0,I_{d-1})$, estimate the gradient information matrix by $C \approx  \frac{1}{M}\sum_{k=1}^{M} \nabla g^\mathrm{PI}(\bm{z}_{-1}^{(k)})\nabla g^\mathrm{PI}(\bm{z}_{-1}^{(k)})^{\mathsf{T}}$;
\STATE (5) Perform the eigendecomposition $C=Q\Lambda Q^{\mathsf{T}}$, and define \(g^\mathrm{PIA}(\bm{z}_{-1}) = g^\mathrm{PI}(Q\bm{z}_{-1})\);
\STATE (6) Evaluate the integral $\mathbb{E}\!\left(g^\mathrm{PIA}(\bm{Z}_{-1})\right)$ by QMC, where \(\bm{Z}_{-1}\sim N(\bm 0,I_{d-1})\).
\end{algorithmic}
\end{algorithm}

\section{Importance Sampling and Active Subspace}
\label{sec:ISAS}

In this section, we propose a three-step method. The method  seamlessly integrates IS with the AS\_PI method, with the aim of avoiding the empirical failure in constructing active subspaces when
dealing with deep out-of-the-money options, as mentioned at the end of Section~\ref{sec:223ASandPreint}. Our three-step method is named the IS\_AS\_PI method.

We refer to the active subspace of $g^\mathrm{I}(\bm{z})$, the new integrand obtained by IS, as the IS-active subspace. We further introduce the concept of orthogonal invariance of active subspaces and demonstrate that this orthogonal invariance holds for both active subspaces and IS-active subspaces. As a consequence, it suffices to consider the standard construction of Brownian motion when applying the three-step method to problems driven by Brownian motion in computational finance.

\subsection{The Proposed Three-Step Method}
\label{subsec:Detailsof3StepMethod}
At the first step, we use IS. For the integrand $g(\bm{z})$ in the expectation $\mathbb{E}(g(\bm{Z}))$, assume $g(\bm{z})\ge 0$. When optimal drift IS is used, the optimal drift $\bm{\mu}^*$ is obtained by solving \eqref{eq:c3ISODISmiu} on \(D_+=\{\bm{z}:g(\bm{z})>0\}\). When Laplace IS is applied, the drift $\bm{\mu}^*$ is computed first, followed by the evaluation of the covariance matrix $\Gamma^*$ via \eqref{eq:c3ISLapISGamma}. Consequently, the problem is transformed into computing the expectation of $g^\mathrm{I}(\bm{z})$ defined in \eqref{eq:c3ISODISintegrand} or \eqref{eq:c3ISLapISintegrand}.
For signed integrands, extending the proposed three-step method using the decomposition in \eqref{c3:SignedIS} is left for future work.

At the second step, the active subspace is computed. We first evaluate the gradient $\nabla g^\mathrm{I}(\bm{z})$. If optimal drift IS is used, the gradient is given by
\begin{equation}
\begin{aligned}
\nabla g^\mathrm{I}(\bm{z})
=\left(
\nabla g(\bm{z}+\bm{\mu}^*)
-g(\bm{z}+\bm{\mu}^*)\bm{\mu}^*
\right)
\exp\!\left(
-{\bm{\mu}^*}^{\mathsf{T}}\bm{z}
-\tfrac{1}{2}{\bm{\mu}^*}^{\mathsf{T}}\bm{\mu}^*
\right).
\end{aligned}
\label{eq:c3ISODISgradient}
\end{equation}
If Laplace IS is used, the gradient is given by
\begin{equation}
\begin{aligned}
\nabla g^\mathrm{I}(\bm{z})
=&\left(
{L^*}^{\mathsf{T}}\nabla g(L^*\bm{z}+\bm{\mu}^*) + g(L^*\bm{z}+\bm{\mu}^*)
\left(
\bm{z}-{L^*}^{\mathsf{T}}(L^*\bm{z}+\bm{\mu}^*)
\right)
\right) \\
&\times
\sqrt{\det(\Gamma^*)}\,
\exp\!\left(
\tfrac12\bm{z}^{\mathsf{T}}\bm{z}
-\tfrac12(L^*\bm{z}+\bm{\mu}^*)^{\mathsf{T}}
(L^*\bm{z}+\bm{\mu}^*)
\right),
\end{aligned}
\label{eq:c3ISLapISgradient}
\end{equation}
where $ L^*{L^*}^{\mathsf{T}}=\Gamma^*$.
In some problems, the gradient of \(g^\mathrm{I}(\bm{z})\) can also be approximated using finite differences \cite{liuPreintegrationActiveSubspace2023}.
Then, the gradient information matrix \(C=\mathbb{E}\!\left(\nabla g^\mathrm{I}(\bm{Z})\nabla g^\mathrm{I}(\bm{Z})^{\mathsf{T}}\right)\) can be estimated by QMC.
Given QMC points $\{\bm{z}^{(1)},\dots,\bm{z}^{(M)}\}$
under the standard Gaussian measure, $C$ 
is approximated by
\begin{equation}
C \approx  \frac{1}{M}\sum_{k=1}^{M}
\nabla g^\mathrm{I}(\bm{z}^{(k)})\nabla g^\mathrm{I}(\bm{z}^{(k)})^{\mathsf{T}}.
\label{eq:c3estimateC}
\end{equation}
Following \eqref{eq:c3eigenC}, we perform the eigendecomposition $C=Q\Lambda Q^{\mathsf{T}}$, 
where \(Q\) is orthogonal.
Define $g^\mathrm{IA}(\bm{z})$ as
\begin{align}
  g^\mathrm{IA}(\bm{z})
  &:=g^\mathrm{I}(Q\bm{z}) .
  \label{eq:c3myASfunction}
\end{align}
The $d$ variables of $g^\mathrm{IA}(\bm{z})$ are ranked in descending order of their estimated mean-squared partial derivatives.

At the third step, we preintegrate with respect to the first active variable. Define 
\begin{align}
  g^\mathrm{IAP}(\bm{z}_{-1})
  &:=\mathbb{E}\!\left(g^\mathrm{IA}(\bm{Z})\mid \bm{Z}_{-1}=\bm{z}_{-1}\right).
  \label{eq:c3mypreintfunction}
\end{align}
If no closed-form expression is available, its values
can be approximated by numerical quadrature as discussed
in Section~\ref{sec:ch2SmoothingEffect}.
Finally, the original problem reduces to computing the expectation $\mathbb{E}\!\left(g^\mathrm{IAP}(\bm{Z}_{-1})\right)$, where \(\bm{Z}_{-1}\sim N(\bm 0,I_{d-1})\).


The three-step procedure is summarized in
Algorithm~\ref{al:ISASpreint}. 
This algorithm is based on implementing IS, the AS method and preintegration in the specific order “IS, the AS method, preintegration”. We adopt this order for two reasons. First, since the AS method relies on the gradient information matrix, IS preceding it can avoid the empirical failure of gradient information matrix estimation in the scenarios of deep out-of-the-money options, as discussed in Section~\ref{sec:223ASandPreint}. Its effectiveness is assessed numerically.
Second, prior research \cite{liuPreintegrationActiveSubspace2023} shows that the AS method followed by preintegration yields better performance than preintegration followed by GPCA in certain cases. Therefore, within our three-step method, we kept the AS method before preintegration to retain this advantage. 
We considered all six possible orderings of the three steps.
The PI\_IS\_GPCA method uses one of these
orderings, as seen in Algorithm~\ref{al:preintISGPCA}.
We excluded the two orderings beginning with the AS method
because directly constructing the active subspace may suffer
from the empirical failure.
In addition to the proposed IS\_AS\_PI method,
we tested PI\_AS\_IS and IS\_PI\_AS.
Neither performed as well as the proposed method in the
Asian option pricing cases considered.

\begin{algorithm}[!tb]
  \caption{IS\_AS\_PI}
  \label{al:ISASpreint}
  \small
  \begin{algorithmic}
    \STATE (1) For the $d$-dimensional function $g(\bm{z})$, compute the optimal drift $\bm{\mu}^*$ via \eqref{eq:c3ISODISmiu};
    \STATE\hspace*{1em} -- If optimal drift IS is used, define $g^\mathrm{I}(\bm{z})$ using \eqref{eq:c3ISODISintegrand};
    \STATE\hspace*{1em} -- If Laplace IS is used, compute $\Gamma^*$ using~\eqref{eq:c3ISLapISGamma} and its Cholesky factor $L^*$; define $g^\mathrm{I}(\bm{z})$ using \eqref{eq:c3ISLapISintegrand};
    \STATE (2) Compute the gradient vector $\nabla g^\mathrm{I}(\bm{z})$;
    \STATE (3) Estimate the gradient information matrix $C$ using \eqref{eq:c3estimateC};
    \STATE (4) Perform the eigendecomposition $C=Q\Lambda Q^{\mathsf{T}}$, and obtain $g^\mathrm{IA}(\bm{z})$ via \eqref{eq:c3myASfunction};
    \STATE (5) Compute the preintegrated function $g^\mathrm{IAP}(\bm{z}_{-1})$ via \eqref{eq:c3mypreintfunction};
    \STATE (6) Use QMC to evaluate the expectation $\mathbb{E}\!\left(g^\mathrm{IAP}(\bm{Z}_{-1})\right)$, where \(\bm{Z}_{-1}\sim N(\bm 0,I_{d-1})\).
  \end{algorithmic}
\end{algorithm}

After IS is applied, the new integrand includes the density ratio \(\varphi(\bm z)/h(\bm z)\), as shown in \eqref{eq:c3ISprinciple}. Consequently, both the affine transformation and this density ratio can affect the gradient information matrix and hence the active subspace. Theoretical analysis of their effects on the variance of the final integrand is left for future work.

\subsection{Orthogonal Invariance}
\label{sec:3invarianceAS}



In this paper, the orthogonal invariance property of active subspaces refers to the property that regardless of the choice of the orthonormal basis in $\mathbb{R}^d$, the active subspace of a certain function can remain the same by choosing the corresponding eigenvectors.
We establish this property under the standard normal density
for both active subspaces and IS-active subspaces obtained
using optimal drift IS or Laplace IS.
Throughout this subsection, all gradient information matrices
are assumed to be finite.
When optimal drift IS or Laplace IS is used, we assume that
equation~\eqref{eq:c3ISODISmiu} has a unique solution for
the optimal drift.
For Laplace IS, we also impose the assumptions stated in
Section~\ref{sec:lapIS}.
We begin with the orthogonal invariance of active subspaces.

\begin{theorem}
  Let $g_1(\bm{x})$ be differentiable almost everywhere, and let $U$ be an orthogonal matrix. 
  Define the coordinate transformation $\bm{y}=U^{\mathsf{T}}\bm{x}$, and set $g_2(\bm{y}) := g_1(U\bm{y})$. 
  Let $C_1$ and $C_2$ denote the gradient information matrices under the standard Gaussian density of $g_1(\bm{x})$ and $g_2(\bm{y})$, respectively.
  Let \(\lambda_1\ge\cdots\ge\lambda_d\) be the eigenvalues of \(C_1\)
  and \(C_1=P\Lambda P^{\mathsf{T}}\) be an eigendecomposition,
  where \(\Lambda=\mathrm{diag}(\lambda_1,\ldots,\lambda_d)\).
  Choose 
  \begin{equation}
    Q=U^{\mathsf{T}}P.
    \label{eq:c3corespondingchoice}
  \end{equation}
  With this choice, we have
  
  (1) \(C_2=Q\Lambda Q^{\mathsf{T}}\) is an eigendecomposition of \(C_2\). 
  
  (2) \(\forall 1\le r\le d\), $g_1(\bm{x})$ and $g_2(\bm{y})$ have the same \(r\)-dimensional active subspace.

  (3) $g_1^\mathrm{A}(\bm z):=g_1(P\bm z)$ and $g_2^\mathrm{A}(\bm z):=g_2(Q\bm z)$ satisfy  \(g_1^\mathrm{A}(\bm z)=g_2^\mathrm{A}(\bm z).\)
  \label{thm:c3ivarrianceofAS} 
\end{theorem}
\begin{proof}
  See Appendix~\ref{prf:thm:c3ivarrianceofAS} for the proof.
\end{proof}

Theorem~\ref{thm:c3ivarrianceofAS} follows from the property of linear subspaces and the orthogonal invariance of the standard normal distribution, which yields \(C_2=U^{\mathsf{T}} C_1 U\).
Theorem~\ref{thm:c3ivarrianceofAS} is similar to the notion of statistical equivalence \cite{zhangEfficientImportanceSampling2021}, but differs in perspective. 
Statistical equivalence is defined in terms of
orthogonally transformed finite samples and implies
that the derived functions $g^\mathrm{A}_1(\bm{z})$ and $g^\mathrm{A}_2(\bm{z})$ have the same expression.
Theorem~\ref{thm:c3ivarrianceofAS} uses active subspaces
defined through expectations and also establishes
that $g^\mathrm{A}_1(\bm{z})$ and $g^\mathrm{A}_2(\bm{z})$ have the same expression.


We introduce the following lemma before establishing
the orthogonal invariance of IS-active subspaces.

\begin{lemma}
  Assume \(\Sigma\) is a positive definite matrix. Let \(R_1,R_2\in\mathbb R^{d\times d}\) satisfy $R_1R_1^{\mathsf{T}}=\Sigma$ and $R_2R_2^{\mathsf{T}}=\Sigma$. Then there exists an orthogonal matrix $U$ such that $R_2=R_1U$.
  \label{lem:c3lemma1}
\end{lemma}
\begin{proof}
  See Appendix~\ref{prf:lem:c3lemma1} for the proof.
\end{proof}

\begin{theorem}
  Let \(g_1(\bm x)\) be a nonnegative function that is
  differentiable almost everywhere, and let \(U\) be an
  orthogonal matrix.
  Define the coordinate transformation $\bm{y}=U^{\mathsf{T}}\bm{x}$, and set $g_2(\bm{y}) := g_1(U\bm{y})$. 
  Let \(g_1^\mathrm{I}(\bm x)\) and \(g_2^\mathrm{I}(\bm y)\)
  denote the integrands obtained by applying either optimal
  drift IS or Laplace IS to both \(g_1(\bm x)\) and
  \(g_2(\bm y)\), respectively.
  Then there exists an orthogonal matrix \(V\) such that
  \[
    g_2^\mathrm{I}(\bm y)=g_1^\mathrm{I}(V\bm y).
  \]
  Let $C_1$ and $C_2$ denote the gradient information matrices under the standard Gaussian density of $g_1^\mathrm{I}(\bm x)$ and $g_2^\mathrm{I}(\bm y)$, respectively.
  Let \(\lambda_1\ge\cdots\ge\lambda_d\) be the eigenvalues of \(C_1\)
  and \(C_1=P\Lambda P^{\mathsf{T}}\) be an eigendecomposition,
  where \(\Lambda=\mathrm{diag}(\lambda_1,\ldots,\lambda_d)\).
  Choose 
  \begin{equation}
    Q=V^{\mathsf{T}}P.
    \label{eq:c3corespondingchoice2}
  \end{equation}
  With this choice, we have
  
  (1) \(C_2=Q\Lambda Q^{\mathsf{T}}\) is an eigendecomposition of \(C_2\). 
  
  (2) \(\forall 1\le r\le d\), $g_1(\bm x)$ and $g_2(\bm y)$ have the same \(r\)-dimensional IS-active subspace.

  (3) \(g_1^\mathrm{IA}(\bm z):=g_1^\mathrm{I}(P\bm z)\)
  and \(g_2^\mathrm{IA}(\bm z):=g_2^\mathrm{I}(Q\bm z)\)
  satisfy
  \(
    g_1^\mathrm{IA}(\bm z)=g_2^\mathrm{IA}(\bm z).
  \)
  \label{thm:c3ivarrianceofISAS}
\end{theorem}
\begin{proof}
  See Appendix~\ref{prf:thm:c3ivarrianceofISAS} for the proof.
\end{proof}

We now apply orthogonal invariance to problems driven by Brownian motion and prove that the choice of the generation matrix of the Brownian motion does not affect the resulting integrand in our three-step method. We refer to Section~\ref{sec:4Application} for the concept of generation matrix of Brownian motion.

Let \(R_1\) and \(R_2\) be two generation matrices of the
same Brownian motion, and define
\(g_1(\bm x)=f(R_1\bm x)\) and
\(g_2(\bm y)=f(R_2\bm y)\).
By Lemma~\ref{lem:c3lemma1}, \(R_2=R_1U\) for some
orthogonal matrix \(U\), so
\(g_2(\bm y)=g_1(U\bm y)\).
Under the assumptions of Theorem~\ref{thm:c3ivarrianceofAS},
choosing the eigenvectors according to
\eqref{eq:c3corespondingchoice} gives
\(
g_1^\mathrm{A}(\bm z)=g_2^\mathrm{A}(\bm z).
\)
Applying the same preintegration step yields identical integrands for the two generation matrices.
Similarly, under the assumptions of
Theorem~\ref{thm:c3ivarrianceofISAS},
choosing the eigenvectors according to
\eqref{eq:c3corespondingchoice2} gives
\(
g_1^\mathrm{IA}(\bm z)=g_2^\mathrm{IA}(\bm z).
\)
Again, applying the same preintegration step yields identical integrands for the two generation matrices.
Therefore, it suffices to consider the standard Brownian
motion construction when using the AS\_PI
or IS\_AS\_PI method.
For the PI\_GPCA and PI\_IS\_GPCA methods, the same conclusion applies to the remaining variables when the
\((d-1)\)-dimensional preintegrated integrand is fixed.
Therefore, in the numerical experiments in Section~\ref{sec:Experiment}, we use only the standard construction of Brownian motion to compare these four methods.

\section{Applications to Finance}
\label{sec:4Application}
In this section, we consider some option pricing problems and sensitivity analysis problems based on the Black--Scholes (BS) model. 
We begin by describing the generation of stock prices, which relies on simulating Brownian motion paths.

Let $S(t)$ denote the stock price at time $t$, $\sigma$ the constant volatility, and $r$ the risk-free interest rate. Let $B(t)$ be a one-dimensional standard Brownian motion. 
Under the risk-neutral measure, the stock price $S(t)$ satisfies the stochastic differential equation (SDE)
\begin{equation}
  \mathrm{d} S(t)=rS(t)\,\mathrm{d} t+\sigma S(t)\mathrm{d} B(t).
  \label{eq:c2SDEofAsset}
\end{equation}
For the initial price $S_0$, the SDE has the analytical solution
\begin{equation}
  S(t)=S_0\exp\left(\left(r-\tfrac{\sigma^2}{2}\right)t+\sigma B(t)\right).
  \label{eq:c2solutionofSDE}
\end{equation}

In \eqref{eq:c2solutionofSDE}, the construction of Brownian motion paths is required because the path of the underlying asset needs to be known. 
Let $0=t_0<t_1<\dots<t_d=T$, where \(T\) is the maturity of the option, and assume the time steps are equally spaced with step size $\Delta t = \frac{T}{d}$, so that $t_j=j\Delta t$ for $j=1,2,\dots,d$.  
For simplicity, write $S_j=S(t_j)$ and $B_j=B(t_j)$. Then 
\begin{equation}
  S_j=S_0\exp\left(\left(r-\tfrac{\sigma^2}{2}\right)j\Delta t+\sigma B_j\right).
  \label{eq:c2solutionofSDEdiscrete}
\end{equation}
The properties of Brownian motion give
\( \bm{B}:=(B_1,B_2,\cdots,B_d)^{\mathsf{T}} \sim  N\!\left(\bm{0},\Sigma \right)\),
where \(\Sigma_{[i,j]} =  \Delta t \min (i,j)\).
For any matrix factorization $\Sigma = RR^{\mathsf{T}}$ and $\bm{Z}\sim N(\bm{0},I_d)$, we have $R\bm{Z}\sim N(\bm{0},\Sigma)$. 
Consequently, a Brownian motion path can be simulated by $\bm{B} = R\bm{Z}$, and $R$ is called the generation matrix of the Brownian motion. Therefore,
\begin{align}
  S_j 
  = S_0 \exp\!\left(\left(r-\tfrac{\sigma^2}{2}\right)j\Delta t+\sigma R_{[j,:]}\bm{z}\right),
  \label{eq:c2generalSj}
\end{align}
where $R_{[j,:]}$ denotes the $j$-th row of $R$.

If $\Sigma =LL^{\mathsf{T}}$ is the Cholesky factorization, then the generation matrix $R_{\mathrm{STD}} = L$ yields the standard construction of Brownian motion, where   
\begin{equation}
  R_{\mathrm{STD}}=\sqrt{\Delta t}
  \begin{bmatrix}
      1 & 0 & \cdots & 0\\
      1 & 1 & \cdots & 0\\
      \vdots & \vdots & \ddots & \vdots\\
      1 & 1 & \cdots & 1
  \end{bmatrix}.
  \label{eq:c2BrownianSTD}
\end{equation}
If $\Sigma = P D P^{\mathsf{T}}$ is the eigendecomposition, where $P$ is orthogonal and $D$ is a diagonal matrix containing eigenvalues in decreasing order, the generation matrix can be taken as $R_{\mathrm{PCA}} = P D^{1/2}$, which gives the PCA construction of Brownian motion \cite{acworthComparisonMonteCarlo1998a}. Moreover, Brownian motion can be constructed using the BB method whose generation matrix is more complicated and may be written analytically in some special cases \cite{wangEffectsDimensionReduction2006}. 

\subsection{Asian Option}
\label{subsec:AsianOptionPricing}
An Asian option is a path-dependent option.  
Let the strike price be $K$ 
and define \(\overline{S} = \overline{S}(\bm{z}) = \frac{1}{d}\sum_{j=1}^{d} S_j\), 
where \(S_j\) is given by \eqref{eq:c2generalSj}.
The payoff of an Asian call option is
\begin{equation}
  g(\bm{z}) := f(R\bm{z})
  := \left(\overline{S}-K\right)_+,
  \label{eq:c2Aisiandiscontpayoff}
\end{equation}
where $(x)_+=x\bm{1}{\{x>0\}}$ and $\bm{1}{\{x>0\}}$ is the indicator function.  
Thus the problem converts to evaluating the integral
\(c = \exp(-rT)\mathbb{E}(g(\bm{Z}))\), where \(\bm Z\sim N(\bm0,I_{d})\).

\subsubsection{Solving for Optimal Drift}
For the payoff function of the Asian option given in~\eqref{eq:c2Aisiandiscontpayoff}, let \(y = \overline{S} \). According to equation~\eqref{eq:c3ISODISmiu}, when \(y - K > 0\), the following equations need to be solved
\begin{equation}
z_i = \frac{1}{g}\frac{\partial g}{\partial z_i} =\frac{\sigma}{(y-K) d} \sum_{j=1}^{d} S_j R_{[j,i]}, \quad i = 1, 2, \dots, d.
\label{eq:c4AsianMiu_i}
\end{equation}
Define \(h_j=\log\frac{S_j}{S_{j-1}}, j=1,2,\cdots,d\), where \(S_j\) is given by \eqref{eq:c2generalSj}. 
Then, replacing \(z_i\) in \(h_j\) with the expressions in~\eqref{eq:c4AsianMiu_i}, we obtain
\begin{equation}
  h_j=\left(r-\tfrac{\sigma^2}{2}\right)\Delta t +  \frac{\sigma^2\Delta t}{(y-K)d}\sum_{i=j}^{d}S_i,\quad j=1,\ldots,d.
\end{equation}
These relations follow from
\(RR^{\mathsf{T}}=\Sigma\) and the identity 
\(\min(j,i)-\min(j-1,i)
=\bm{1}_{\{i\ge j\}}\).
Then, recurrence relations are given by
\begin{equation}
  \begin{aligned}
    h_1&=\left(r-\tfrac{\sigma^2}{2}\right)\Delta t +\frac{y\sigma^2\Delta t}{y-K}, \\
   S_j&=S_{j-1}\exp({h_j}),\quad j=1,\ldots,d, \\
   h_{j+1}&=h_j-\frac{\sigma^2\Delta t}{(y-K)d}S_j, \quad j=1,\ldots,d-1. 
   \label{eq:c4iteration_Asian}
  \end{aligned}
\end{equation}
For a given value \(y>K\),  these recurrences uniquely determine a price path.
Therefore, \(\overline{S}\) can be regarded as a function of \(y\), i.e., \(\overline{S} = \overline{S}(y)\). Hence, equations~\eqref{eq:c4AsianMiu_i} can be transformed into a scalar equation:
\begin{equation}
\overline{S}(y) - y = 0.
\label{eq:c4AsianScalarEqu}
\end{equation}
By solving the equation~\eqref{eq:c4AsianScalarEqu}, 
we can obtain the optimal drift $\bm\mu^*$ using~\eqref{eq:c4AsianMiu_i}.
If Laplace IS is used, the covariance matrix in \eqref{eq:c3ISLapISGamma} is 
\[
\Gamma^*
=
\left[
I_d+\bm{\mu}^*{\bm{\mu}^*}^{\mathsf{T}}
-
\frac{\sigma^2}{\left(\overline{S}(\bm\mu^*)-K\right)d}
R^{\mathsf{T}}
\operatorname{diag}(S_1(\bm\mu^*),\ldots,S_d(\bm\mu^*))
R
\right]^{-1}.
\]
Let \(L^*\) satisfy \(L^*{L^*}^{\mathsf{T}}=\Gamma^*\).

\subsubsection{Deriving the New Integrand}
\label{subsec:AsianOptionPricing-2}

The transformed integrand \(g^\mathrm{I}(\bm z)\) is given by~\eqref{eq:c3ISODISintegrand} and~\eqref{eq:c3ISLapISintegrand} for optimal drift IS and Laplace IS, respectively.
And the gradient of \(g^\mathrm{I}(\bm z)\) can be obtained by~\eqref{eq:c3ISODISgradient} and \eqref{eq:c3ISLapISgradient} analytically.
The gradient information matrix \(C\) can be estimated according to~\eqref{eq:c3estimateC}, and can be decomposed as \(C = Q \Lambda Q^{\mathsf{T}}\). The active subspace of the importance sampled integrand \(g^\mathrm{I}(\bm{z})\) is thus obtained. Then \(g^\mathrm{IA}(\bm{z})\) can be obtained by \eqref{eq:c3myASfunction}. 
If optimal drift IS is used, the corresponding variable transformation is performed as
\begin{equation}
\begin{aligned}
g^\mathrm{IA}(\bm{z})
=&\left(\frac{S_0}{d} \sum_{j=1}^{d} \exp\left(\left(r - \frac{\sigma^2}{2}\right) j \Delta t + \sigma R_{[j,:]}(Q \bm{z} + \bm{\mu}^*)\right) - K \right)_+ \\
&\times \exp\left(-\tfrac{1}{2} {\bm{\mu}^*}^{\mathsf{T}} \bm{\mu}^* - {\bm{\mu}^*}^{\mathsf{T}} Q \bm{z}\right).
\end{aligned}
\label{eq:c4myISASfunction}
\end{equation}
If Laplace IS is used, the transformation is performed as 
\begin{equation}
\begin{aligned}
g^\mathrm{IA}(\bm{z})
=&\left(\frac{S_0}{d} \sum_{j=1}^{d}
\exp\left(
\left(r-\frac{\sigma^2}{2}\right)j\Delta t
+\sigma R_{[j,:]}(L^*Q\bm{z}+\bm{\mu}^*)
\right)-K\right)_+ \\
&\times \sqrt{\det(\Gamma^*)}
\exp\left(
\tfrac{1}{2}\bm{z}^{\mathsf{T}}\bm{z}
-\tfrac{1}{2}(L^*Q\bm{z}+\bm{\mu}^*)^{\mathsf{T}}
(L^*Q\bm{z}+\bm{\mu}^*)
\right).
\end{aligned}
\label{eq:c4myLapISASfunction}
\end{equation}

\subsubsection{Preintegration}
\label{subsec:AsianOptionPricing-3}
Let \(\mathcal A=\mathcal A(\bm z_{-1})\) denote the
preintegration region. In this problem, the closed form of  \(g^\mathrm{IAP}(\bm z_{-1})\) can be obtained, so we analyze the preintegration region first. 
For optimal drift IS, note 
\[k(z_1,\bm{z}_{-1}) = \frac{S_0}{d} \sum_{j=1}^{d} \exp\left(\left(r - \frac{\sigma^2}{2}\right) j \Delta t + \sigma R_{[j,:]}(Q \bm{z} + \bm{\mu}^*)\right) - K. \] 
Then \(\bm{1}\{k(z_1,\bm z_{-1})>0\}\) is the discontinuous factor in \eqref{eq:c4myISASfunction}. 
If the first column of \(RQ\) is componentwise nonpositive,
we replace \(Q\) by \(-Q\).
If the first column of \(RQ\) is componentwise nonnegative,
\(k(z_1,\bm{z}_{-1})\) is strictly increasing in \(z_1\)
for each fixed \(\bm z_{-1}\). 
In this case, the equation \(k(z_1,\bm z_{-1})=0\) has no root or one unique root.
If the equation has no root, preintegration can be performed over \(\mathcal{A}=\mathbb{R}\).
If the equation has a unique root $\gamma=\gamma(\bm{z}_{-1})$, the discontinuous factor can be written as \(\bm{1}\{k(z_1,\bm{z}_{-1})>0\} = \bm{1}\{z_1>\gamma\}\). Therefore, the preintegration region is \(\mathcal{A}=(\gamma,+\infty)\).
If the first column of \(RQ\) contains both positive and negative
entries, then, for each fixed \(\bm z_{-1}\),
\(k(z_1,\bm z_{-1})\) is strictly convex in \(z_1\), since
\(\frac{\partial^2}{\partial z_1^2}k(\bm z)>0\).
Moreover, \(k(z_1,\bm z_{-1})\to+\infty\) as
\(z_1\to\pm\infty\).
Thus the equation \(k(z_1,\bm{z}_{-1})=0\) can have zero, one or two roots. 
If the equation has zero or one root, preintegration can be performed over \(\mathcal{A}=\mathbb{R}\). 
If the equation has two distinct roots \(\gamma_1<\gamma_2\), the preintegration region is \(\mathcal{A}=(-\infty,\gamma_1)\cup(\gamma_2,\infty)\).
For Laplace IS, the preintegration region can be analyzed in the same way using the first column of \(RL^*Q\). 

For simplicity, we give a unified closed-form expression for \(g^\mathrm{IAP}(\bm z_{-1})\) for both IS methods.
Define
\(
A=Q^{\mathsf{T}}{L^*}^{\mathsf{T}} L^*Q,
\bm{\beta}={\bm{\mu}^*}^{\mathsf{T}} L^*Q,
W=\sigma RL^*Q-\bm{1}_{d\times1}\bm{\beta}
\),
where \(\bm{1}_{d\times1}\) denotes the column vector of ones.
The matrix \(A\) is positive definite, so \(A_{[1,1]}>0\).
Let \(\beta_k\) denote the \(k\)-th component of \(\bm{\beta}\),
and define
\(
\kappa=\kappa(\bm z_{-1})
=\sum_{k=2}^{d}A_{[1,k]}z_k\),
\(\chi=\chi(\bm z_{-1})
=\frac12\sum_{k=2}^{d}z_k^2
-\frac12\sum_{k=2}^{d}\sum_{\ell=2}^{d}
A_{[k,\ell]}z_kz_\ell.
\)
According to \eqref{eq:c3mypreintfunction}, we have
\begin{align}
g^\mathrm{IAP}(\bm z_{-1})
&=
\sqrt{\det(\Gamma^*)}\,
\exp\!\left(
-\tfrac12{\bm{\mu}^*}^{\mathsf{T}}\bm{\mu}^*
\right)
(I_1+I_2),
\label{eq:c4mypreintfunction}
\end{align}
where
\begin{align}
I_1
=\frac{S_0e^\chi}{d\sqrt{A_{[1,1]}}}
\sum_{j=1}^{d}
\exp\!\left(
\left(r-\tfrac{\sigma^2}{2}\right)j\Delta t
+\sigma R_{[j,:]}\bm{\mu}^*
+\sum_{k=2}^{d}W_{[j,k]}z_k
+\frac{(W_{[j,1]}-\kappa)^2}{2A_{[1,1]}}
\right)J_j,\notag
\label{eq:c4mypreintfunctionI1expression}
\end{align}
and 
\begin{align}
I_2
&=-\frac{Ke^\chi}{\sqrt{A_{[1,1]}}}
\exp\!\left(
-\sum_{k=2}^{d}\beta_kz_k
+\frac{(\beta_1+\kappa)^2}{2A_{[1,1]}}
\right)J_0.\notag
\label{eq:c4mypreintfunctionI2expression}
\end{align}
The factors \(J_j\), \(j=1,\ldots,d\), and \(J_0\)
depend on the preintegration region \(\mathcal A\).
For \(\mathcal A=\mathbb R\), we have \(J_j=J_0=1\). For \(\mathcal A=(\gamma,+\infty)\), we have
\[
\begin{aligned}
J_j
=1-\Phi\!\left(
\tfrac{A_{[1,1]}\gamma-W_{[j,1]}+\kappa}
{\sqrt{A_{[1,1]}}}
\right),\quad
J_0
=1-\Phi\!\left(
\tfrac{A_{[1,1]}\gamma+\beta_1+\kappa}
{\sqrt{A_{[1,1]}}}
\right).
\end{aligned}
\]
For \(\mathcal A=(-\infty,\gamma_1)\cup(\gamma_2,+\infty)\),
we have
\[
\begin{aligned}
J_j
&=1+\Phi\!\left(
\tfrac{A_{[1,1]}\gamma_1-W_{[j,1]}+\kappa}
{\sqrt{A_{[1,1]}}}
\right)
-\Phi\!\left(
\tfrac{A_{[1,1]}\gamma_2-W_{[j,1]}+\kappa}
{\sqrt{A_{[1,1]}}}
\right),\\
J_0
&=1+\Phi\!\left(
\tfrac{A_{[1,1]}\gamma_1+\beta_1+\kappa}
{\sqrt{A_{[1,1]}}}
\right)
-\Phi\!\left(
\tfrac{A_{[1,1]}\gamma_2+\beta_1+\kappa}
{\sqrt{A_{[1,1]}}}
\right).
\end{aligned}
\]
The corresponding expressions for optimal drift IS follow by
setting \(L^*=\Gamma^*=I_d\), in which case \(A=I_d\) and
\(\kappa=\chi=0\).
At this point, the explicit expression of $g^\mathrm{IAP}(\bm{z}_{-1})$ in
\eqref{eq:c4mypreintfunction} has been fully derived.
The price of the Asian call option can thus be written as
\( c = \exp(-rT)\,\mathbb{E}\!\left(g^\mathrm{IAP}(\bm{Z}_{-1})\right)\), where \(\bm{Z}_{-1}\sim N(\bm 0,I_{d-1})\).

\subsection{The Delta of Asian Options}
\label{sec:DeltaOfAsianOption}
We consider the estimation of the Delta of Asian options in this subsection.
The Delta of an option is defined as the partial derivative of the option price with respect to the underlying asset price. 
For Asian options, estimating the Delta by the pathwise (PW) estimator \cite{BroadieEstimatingsecurityprice1996} gives
\begin{equation}
    \Delta
    =
    \frac{e^{-rT}}{S_0}
    \mathbb{E}\!\left[
        \overline{S}(\bm Z)
        \bm{1}\{\overline{S}(\bm Z)>K\}
    \right],
    \qquad
    \bm Z\sim N(\bm 0,I_d).
    \label{eq:c4DeltaPW}
\end{equation}
The integrand \(\overline{S}(\bm z)\bm{1}\{\overline{S}(\bm z)>K\}\) is discontinuous.
Although importance sampling can be applied to this integrand,
its ordinary pointwise gradient does not capture the jump.
To account for this discontinuity, 
we apply a substitution described in the 
discontinuity removing method 
\cite{wangHandlingDiscontinuitiesFinancial2016}
before determining the optimal drift and constructing
the active subspace.

Assume that all entries of the first column $R_{[:,1]}$ are positive. Then, for each fixed $\bm z_{-1}$,
$\overline S(z_1,\bm z_{-1})$ is strictly increasing in $z_1$,
with limits $0$ and $+\infty$ as $z_1\to-\infty$ and
$z_1\to+\infty$, respectively.
Hence, for $K>0$, the equation
$\overline S(z_1,\bm z_{-1})-K=0$
has a unique finite solution $\alpha=\alpha(\bm z_{-1})$, so that \(\bm{1}\{\overline{S}(\bm z)>K\}=\bm{1}\{z_1>\alpha(\bm{z}_{-1})\}\).
Define \(\tau = \tau(z_1;\bm z_{-1}) = \Phi^{-1}\!\left( \Phi(\alpha(\bm z_{-1}))+[1-\Phi(\alpha(\bm z_{-1}))]\Phi(z_1)\right)\). 
Replacing \(z_1\) in the integral \eqref{eq:c4DeltaPW} with \(\tau(z_1;\bm z_{-1})\) gives
\begin{equation}
    \Delta
    =
    \frac{e^{-rT}}{S_0}\mathbb{E}[[1-\Phi(\alpha(\bm Z_{-1}))]
    \overline{S}\!\left(
        \tau(Z_1;\bm Z_{-1}),\bm Z_{-1}\right)].
    \label{eq:c4DeltaSmoothExpectation}
\end{equation}
Throughout this subsection, $g(\bm z)$ denotes the transformed integrand
\begin{equation}
    g(\bm z)
    =
    [1-\Phi(\alpha(\bm z_{-1}))]
    \overline{S}\!\left(
        \tau(z_1;\bm z_{-1}),\bm z_{-1}
    \right).
    \label{eq:c4DeltaSmoothIntegrand}
\end{equation}
This transformation preserves the integral exactly and produces
a strictly positive, smooth integrand on $\mathbb R^d$.
The dimension remains $d$ at this stage.

\subsubsection{Solving for Optimal Drift and Deriving the New Integrand}

We construct the optimal drift using the transformed integrand \(g(\bm z)\) in \eqref{eq:c4DeltaSmoothIntegrand}. Both \(\nabla g(\bm z)\) and \(\nabla^2 g(\bm z)\) are available analytically. The optimal drift can be obtained using recurrence relations, and the inverse of the covariance matrix for Laplace IS can be calculated analytically. Then \(g^\mathrm{I}(\bm z)\) can be obtained by~\eqref{eq:c3ISODISintegrand} and~\eqref{eq:c3ISLapISintegrand} for optimal drift IS and Laplace IS, respectively.
The gradient \(\nabla g^\mathrm{I}(\bm z)\) is available analytically. 
Then the gradient information matrix \(C\) can be estimated according to~\eqref{eq:c3estimateC}, and after its eigendecomposition \(g^\mathrm{IA}(\bm{z})\) can be obtained by \eqref{eq:c3myASfunction}.

\subsubsection{Preintegration}
In this problem, we cannot find the closed form of \(g^{\mathrm{IAP}}(\bm z_{-1})\), so we use numerical integration. 
According to \eqref{eq:c3mypreintfunction}, define
\begin{equation}
\begin{aligned}
    g^{\mathrm{IAP}}(\bm z_{-1})
    &=
    \int_{\mathbb R}
    g^{\mathrm{IA}}(z_1,\bm z_{-1})
    \varphi(z_1)\,\mathrm d z_1.
    \label{eq:c4DeltaPreintegrand}
\end{aligned}
\end{equation}
The Delta is given by \(\Delta=\frac{e^{-rT}}{S_0} \mathbb{E}\!\left[g^{\mathrm{IAP}}(\bm Z_{-1}) \right]\), where \(\bm Z_{-1}\sim N(\bm 0,I_{d-1})\). 
For each fixed sample point $\bm z_{-1}$, we evaluate the one-dimensional integral in \eqref{eq:c4DeltaPreintegrand} numerically. The numerical integral provides one value of $g^{\mathrm{IAP}}(\bm z_{-1})$. 
In the numerical experiment in Section~\ref{sec:AsianDeltaConvergence}, we use adaptive Gauss--Kronrod quadrature over $\mathbb R$ by the SciPy function \texttt{scipy.integrate.quad\_vec}.

\subsection{Basket Option}
\label{subsec:BasketOptionPricing}
A basket option depends on multiple underlying stocks. Here, we consider a basket option depending on two correlated stocks \(S^{(\ell)}(t)\), \(\ell=1,2\). For \(\ell=1,2\), \(S^{(\ell)}(t)\) follows the SDE in \eqref{eq:c2SDEofAsset}
with initial price \(S_0^{(\ell)}>0\) and volatility \(\sigma_\ell>0\).
As in \eqref{eq:c2solutionofSDEdiscrete}, we use equally
spaced monitoring times $t_j=j\Delta t$,
$j=1,\ldots,d$, with $\Delta t=T/d$.
We additionally assume that the two Brownian motions have correlation coefficient \(\mathrm{Corr}\left(B^{(1)}(t),B^{(2)}(t)\right)
=\rho\), where \(\rho\in (-1,1)\). Define 
\begin{equation}
\widetilde{\Sigma}
=
\begin{pmatrix}
\sigma_1^2\Sigma & \rho\sigma_1\sigma_2\Sigma\\
\rho\sigma_1\sigma_2\Sigma & \sigma_2^2\Sigma
\end{pmatrix},
\label{eq:c4BasketCovariance}
\end{equation}
where \(\Sigma_{[i,j]}=\Delta t\min(i,j)\). Let \(\widetilde R\in\mathbb R^{2d\times2d}\) be a generation matrix satisfying
\(\widetilde R\widetilde R^{\mathsf{T}}=\widetilde{\Sigma}\).
For \(\bm z\in \mathbb{R}^{2d}\), the asset prices are
\begin{equation}
\begin{aligned}
S_j^{(1)}(\bm z)
&=
S_0^{(1)}
\exp\left(
\left(r-\tfrac{\sigma_1^2}{2}\right)j\Delta t
+\widetilde R_{[j,:]}\bm z
\right),\\
S_j^{(2)}(\bm z)
&=
S_0^{(2)}
\exp\left(
\left(r-\tfrac{\sigma_2^2}{2}\right)j\Delta t
+\widetilde R_{[d+j,:]}\bm z
\right),
\quad j=1,\ldots,d.
\end{aligned}
\label{eq:c4BasketSj}
\end{equation}
The volatilities and correlation coefficient are already included in \(\widetilde R\).
Define
\(
\overline S^{(\ell)}(\bm z)
=
\frac1d\sum_{j=1}^{d}S_j^{(\ell)}(\bm z).
\)
For weights \(w_1,w_2>0\), with \(w_1+w_2=1\), and strike
price \(K>0\), the payoff is
\begin{equation}
g(\bm z)
:=f(\widetilde R\bm z)
:=
\left(
w_1\overline S^{(1)}(\bm z)
+w_2\overline S^{(2)}(\bm z)-K
\right)_+.
\label{eq:c4BasketPayoff}
\end{equation}
Thus the pricing problem is to evaluate
\(c=\exp(-rT)\mathbb E(g(\bm Z))\), where \(\bm Z\sim N(\bm 0,I_{2d})\).


\subsubsection{Solving for Optimal Drift}

For the payoff in~\eqref{eq:c4BasketPayoff}, let
\(y_1=\overline S^{(1)}(\bm z)\),
\(y_2=\overline S^{(2)}(\bm z)\),
and \(H=w_1y_1+w_2y_2-K\).
According to equation~\eqref{eq:c3ISODISmiu}, on the region
\(H>0\), the following equations need to be solved
\begin{equation}
z_i
=
\frac1g\frac{\partial g}{\partial z_i}
=
\frac{1}{dH}
\sum_{j=1}^{d}
\left[
w_1S_j^{(1)}\widetilde R_{[j,i]}
+w_2S_j^{(2)}\widetilde R_{[d+j,i]}
\right],
\quad i=1,\ldots,2d.
\label{eq:c4BasketMiu_i}
\end{equation}
For \(\ell=1,2\), define 
\(h_j^{(\ell)}=\log\frac{S_j^{(\ell)}}{S_{j-1}^{(\ell)}}\), 
\(j=1,\ldots,d\), 
where \(S_j^{(\ell)}\) is given by \eqref{eq:c4BasketSj}.
Then, replacing \(z_i\) in \(h_j^{(\ell)}\) with the expressions in~\eqref{eq:c4BasketMiu_i}, we have
\begin{equation}
\begin{aligned}
h_j^{(1)}
&=
\left(r-\frac{\sigma_1^2}{2}\right)\Delta t
+\frac{\Delta t}{dH}
\sum_{i=j}^{d}
\left[
\sigma_1^2w_1S_i^{(1)}
+\rho\sigma_1\sigma_2w_2S_i^{(2)}
\right],\\
h_j^{(2)}
&=
\left(r-\frac{\sigma_2^2}{2}\right)\Delta t
+\frac{\Delta t}{dH}
\sum_{i=j}^{d}
\left[
\rho\sigma_1\sigma_2w_1S_i^{(1)}
+\sigma_2^2w_2S_i^{(2)}
\right].
\end{aligned}
\label{eq:c4BasketLogIncrements}
\end{equation}
These relations follow from \(\widetilde R\widetilde R^{\mathsf{T}}=\widetilde{\Sigma}\)
and the identity \(\min(j,i)-\min(j-1,i)=\bm1_{\{i\ge j\}}\).
Then, recurrence relations are given by
\begin{equation}
\begin{aligned}
h_1^{(1)}
&=
\left(r-\tfrac{\sigma_1^2}{2}\right)\Delta t
+\frac{\Delta t}{H}
\left(
\sigma_1^2w_1y_1
+\rho\sigma_1\sigma_2w_2y_2
\right),\\
h_1^{(2)}
&=
\left(r-\tfrac{\sigma_2^2}{2}\right)\Delta t
+\frac{\Delta t}{H}
\left(
\rho\sigma_1\sigma_2w_1y_1
+\sigma_2^2w_2y_2
\right),\\
S_j^{(\ell)}
&=
S_{j-1}^{(\ell)}\exp\left(h_j^{(\ell)}\right),
\qquad j=1,\ldots,d,\quad \ell=1,2,\\
h_{j+1}^{(1)}
&=
h_j^{(1)}
-\frac{\Delta t}{dH}
\left(
\sigma_1^2w_1S_j^{(1)}
+\rho\sigma_1\sigma_2w_2S_j^{(2)}
\right),\quad j=1,\ldots,d-1,\\
h_{j+1}^{(2)}
&=
h_j^{(2)}
-\frac{\Delta t}{dH}
\left(
\rho\sigma_1\sigma_2w_1S_j^{(1)}
+\sigma_2^2w_2S_j^{(2)}
\right),
\quad j=1,\ldots,d-1.
\end{aligned}
\label{eq:c4iteration_Basket}
\end{equation}
For a given pair \(y_1,y_2>0\) satisfying \(H>0\), these
recurrences uniquely determine two price paths. 
Therefore, \(\overline S^{(\ell)}\) can be regarded as a function of \(y_1,y_2\), i.e., \(\overline S^{(\ell)} = \overline{S}^{(\ell)}(y_1,y_2)\). 
Hence, solving \eqref{eq:c4BasketMiu_i} reduces to solving
the following system of two equations in $y_1$ and $y_2$:
\begin{equation}
\begin{aligned}
\overline{S}^{(1)}(y_1,y_2)-y_1&=0, \qquad
\overline{S}^{(2)}(y_1,y_2)-y_2&=0.
\end{aligned}
\label{eq:c4BasketScalarEqu}
\end{equation}
By solving equations~\eqref{eq:c4BasketScalarEqu}, 
we can obtain the optimal drift $\bm\mu^*$ using~\eqref{eq:c4BasketMiu_i}. 
If Laplace IS is used, the covariance matrix can be obtained from \eqref{eq:c3ISLapISGamma}. 

\subsubsection{Deriving the New Integrand and Preintegration}

The IS integrand is obtained as in
Section~\ref{subsec:AsianOptionPricing-2},
and the active subspace is constructed from its analytical gradient.
The preintegration region is analyzed as in
Section~\ref{subsec:AsianOptionPricing-3}.
Once the integration boundaries are determined,
the resulting one-dimensional integrals can be evaluated
in closed form.

\section{Numerical Experiments}
\label{sec:Experiment}
In this section, some numerical experiments are given in order to verify the effectiveness of the proposed algorithm. To be clear, 
we use ODIS and LapIS to denote the optimal drift IS and Laplace IS respectively.

In our experiments, we consider the following six methods, unless specified:
\begin{itemize}
  \item[(1)] \textbf{IS\_AS\_PI}: the method proposed in this paper, both ODIS and LapIS are considered;
  \item[(2)] \textbf{PI\_IS\_GPCA} \cite{zhangEfficientImportanceSampling2021}: one of the currently popular methods, both ODIS and LapIS are considered;
  \item[(3)] \textbf{AS\_PI} \cite{liuPreintegrationActiveSubspace2023}: another popular method currently;
  \item[(4)] \textbf{PI\_GPCA} \cite{xiaoConditionalQuasiMonteCarlo2018}: commonly used as a comparative method in other studies;
  \item[(5)] \textbf{RQMC} (standard randomized QMC method);
  \item[(6)] \textbf{MC} (crude MC method): used as a baseline.
\end{itemize}

All methods except MC use scrambled Sobol' point sets generated using the SciPy function \texttt{scipy.stats.qmc.Sobol}, with fixed random seeds to ensure a fair comparison. That is, the methods except MC share the same randomized points within each replication.
In our implementation, all gradients are calculated analytically.
For methods requiring a gradient information matrix,
we estimate it using \(M=128\) RQMC pilot points.
The Brownian motion is constructed by the standard construction, 
whose generation matrix is given in \eqref{eq:c2BrownianSTD}. 
We compute the variance according to \cite{glassermanMonteCarloMethods2003}. And the standard error (SE) is the square root of the variance. The variance reduction factor (VRF) is defined by the ratio of the variance of MC to the variance of the corresponding method \cite{glassermanMonteCarloMethods2003}.


\subsection{Asian Option Pricing}
\label{sec:AsianExperiment}
We evaluate the performance of IS\_AS\_PI in Asian option pricing by experiments.
The following parameters are chosen as in \cite{liuPreintegrationActiveSubspace2023}: $d=50$ time steps (giving the same problem dimension), maturity $T=1$, volatility $\sigma=0.4$, risk-free interest rate $r=0.1$ and initial stock price  $S_0=100$. We consider strike prices $K=60,100,140,180$ in these experiments except in Section~\ref{sec:AsianFailure}.

\subsubsection{Convergence Analysis}
\label{sec:AsianConvergence}
This experiment compares the convergence behavior of the methods above. 
For each sample size $n=2^N$ with $N=5,6,\dots,17$, $m=50$ independent replications are performed to estimate the SEs for methods based on RQMC.
The numerical results are shown in Figure~\ref{fig:asian-convergence}.
In all panels, the vertical axis represents the base-2 logarithm of the SE. Markers represent the estimated \(\log_2 (\mathrm{SE})\), and dashed lines show ordinary least-squares linear fits of \(\log_2 (\mathrm{SE})\) against \(N=\log_2 n\) over \(N=5,\ldots,17\). 
The annotations \(n^k\) summarize the fitted dependence
of SE on \(n\), where \(k\) is the slope of the
corresponding least-squares line.
\begin{figure}[!tb]
    \centering

    \subfloat[$K=60$]{%
        \includegraphics[width=0.48\linewidth]{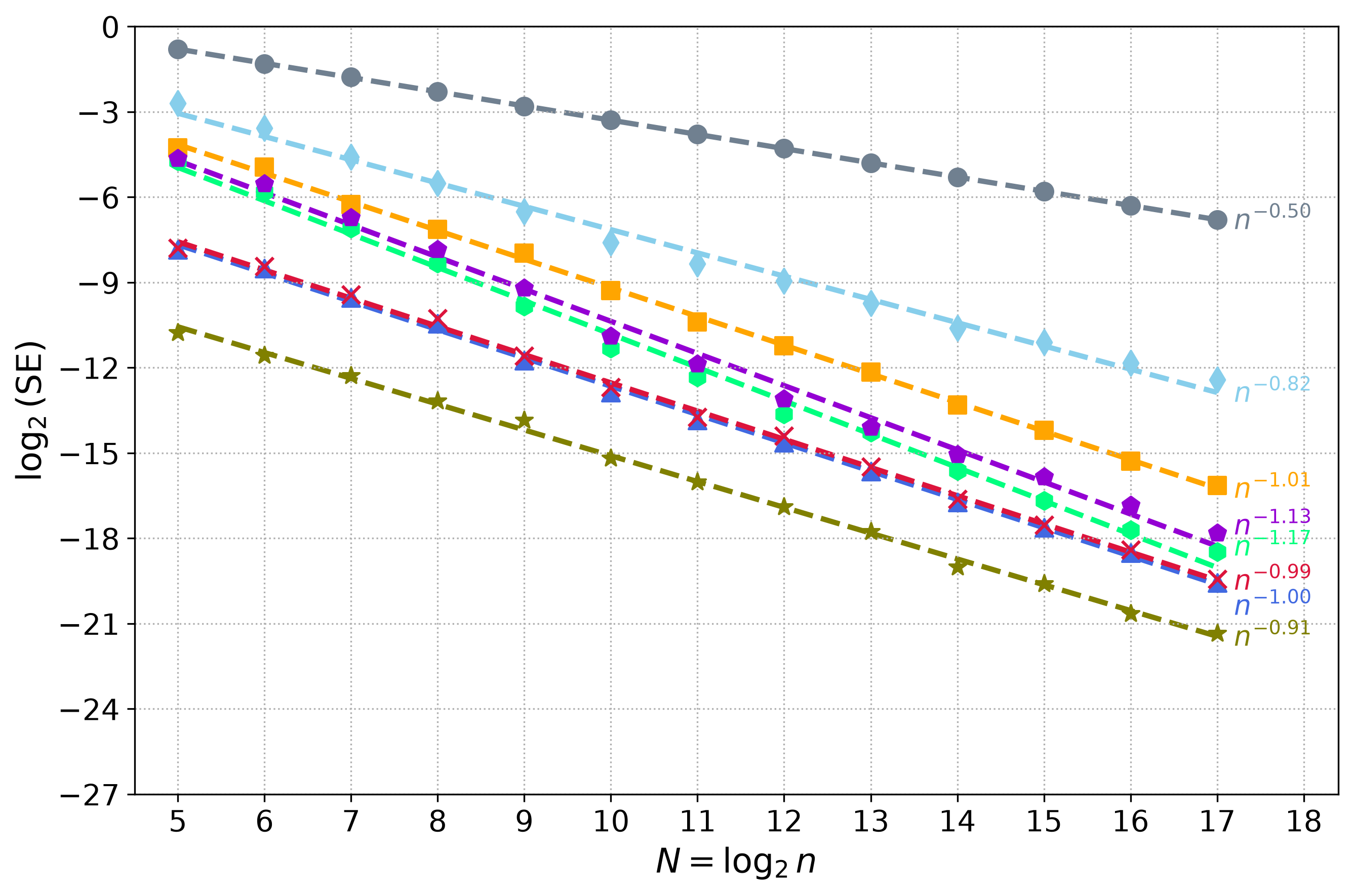}%
        \label{fig:K60}%
    }
    \hfill
    \subfloat[$K=100$]{%
        \includegraphics[width=0.48\linewidth]{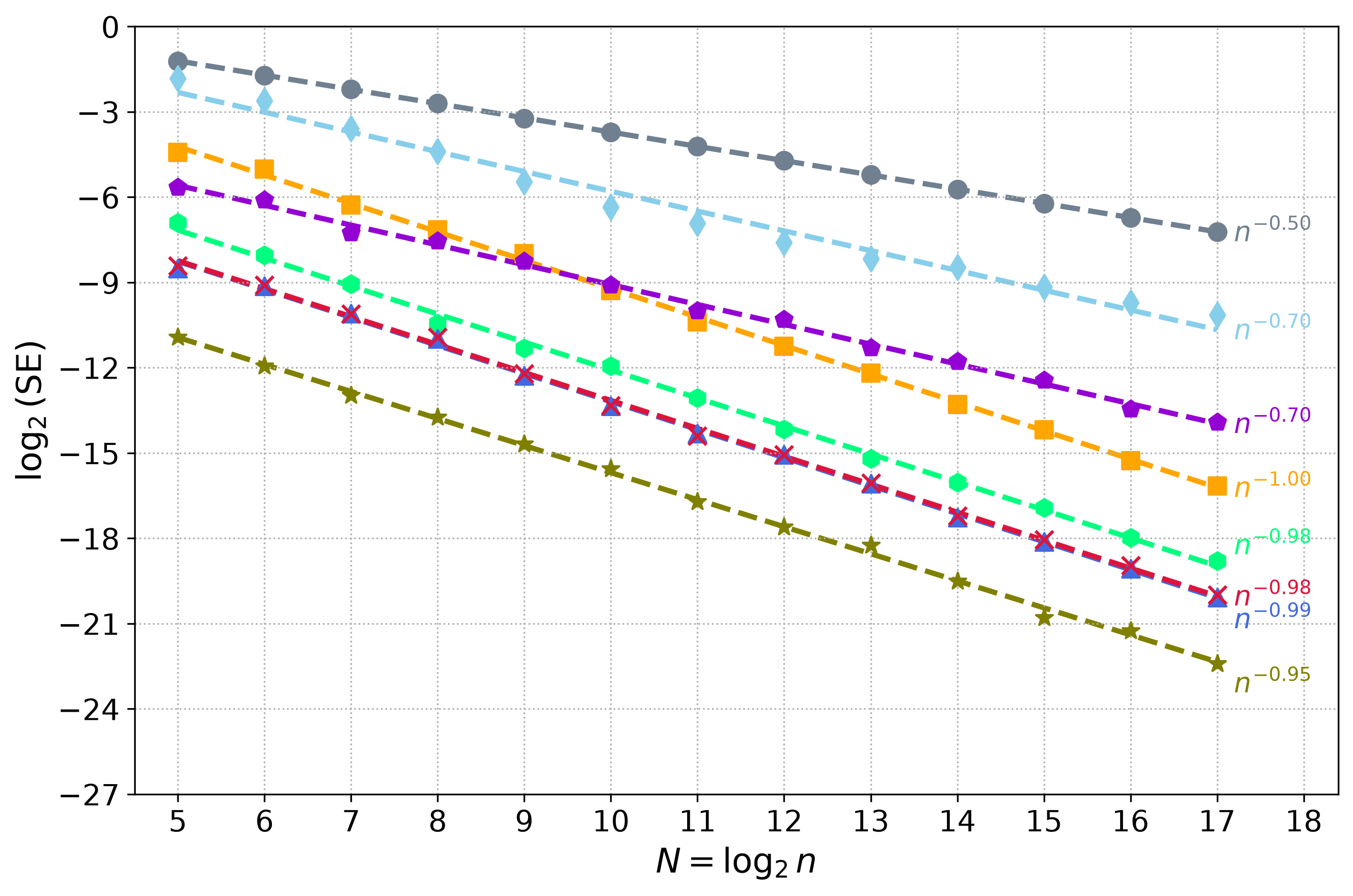}%
        \label{fig:K100}%
    }

    \subfloat[$K=140$]{%
        \includegraphics[width=0.48\linewidth]{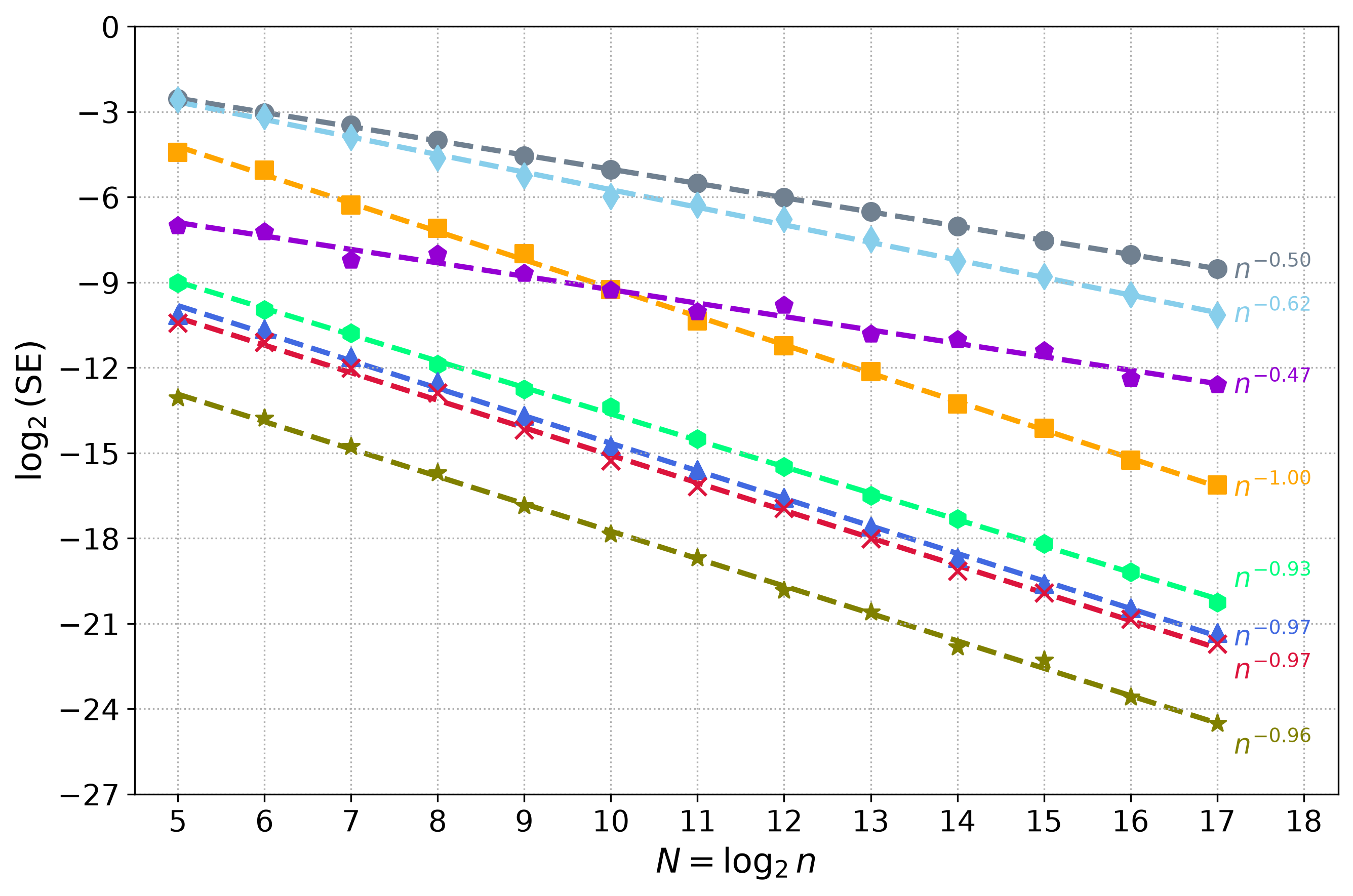}%
        \label{fig:K140}%
    }
    \hfill
    \subfloat[$K=180$]{%
        \includegraphics[width=0.48\linewidth]{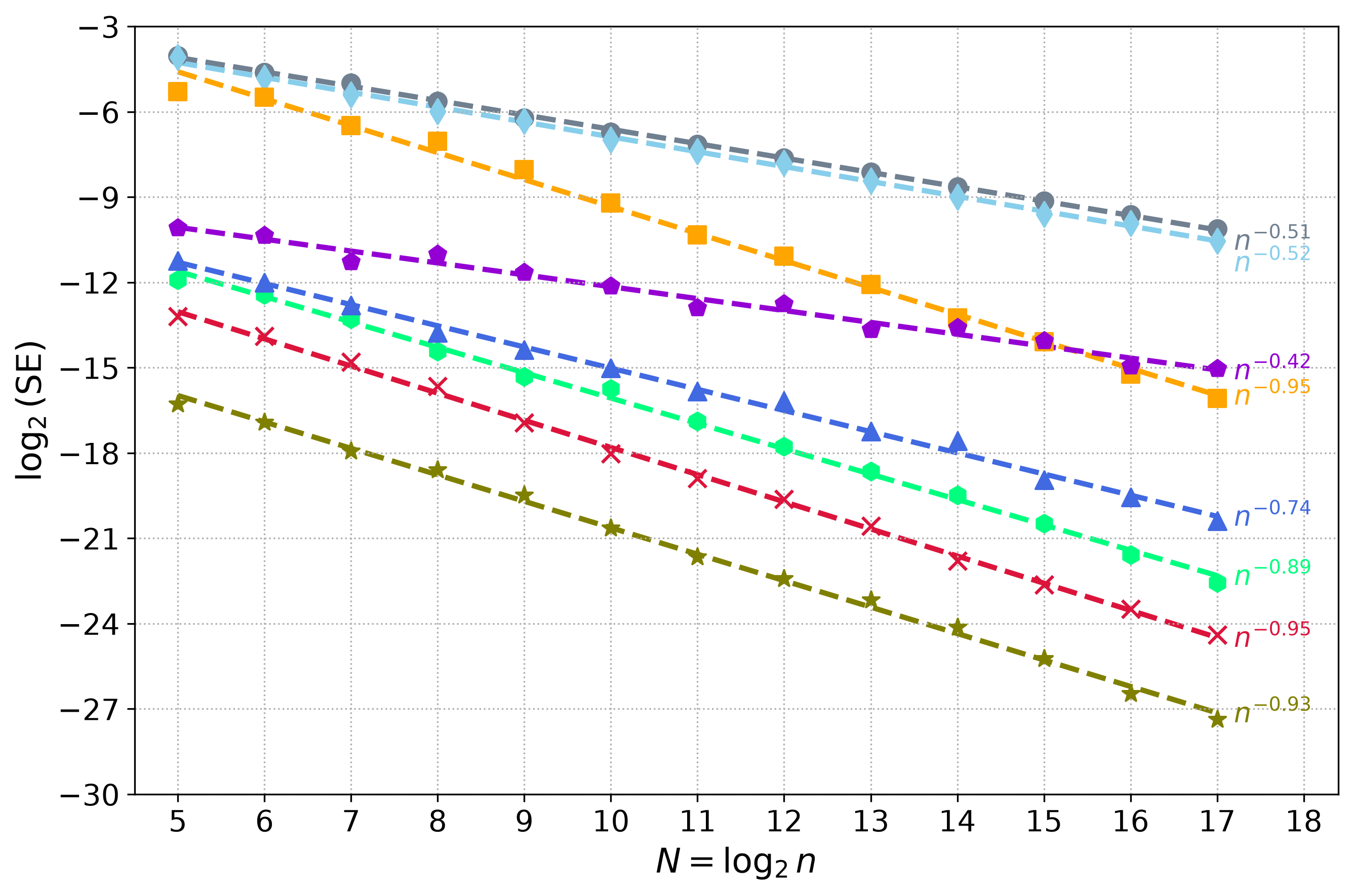}%
        \label{fig:K180}%
    }

    \par\smallskip

    \includegraphics[width=0.7\linewidth]{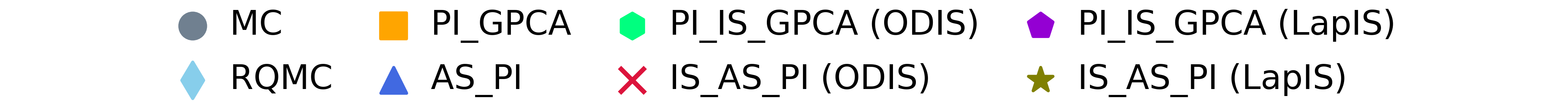}

    \caption{Convergence behavior for Asian option pricing with different strike prices.}
    \label{fig:asian-convergence}
\end{figure}

From Figure~\ref{fig:asian-convergence}, we observe that, at the same sample sizes, LapIS based IS\_AS\_PI achieves the smallest estimated SEs among the methods compared in all settings. 
The performance of ODIS based IS\_AS\_PI is close to the performance of AS\_PI \cite{liuPreintegrationActiveSubspace2023} in Figures~\ref{fig:K60} and \ref{fig:K100}, which shows that IS does not always reduce the variance. Thus ODIS is not necessary in these cases. However, ODIS based IS\_AS\_PI has lower SEs than AS\_PI at the same sample sizes in Figures~\ref{fig:K140} and \ref{fig:K180}. The advantage is more pronounced when $K$ is larger.

For a sample size $n=2^{17}$ with $m=50$ independent replications, VRFs and SEs for the methods are reported in Table~\ref{tab:AsianPricingVRF-STD-m17-p17}. We focus on VRFs in this table, while the SEs support similar observations. LapIS based IS\_AS\_PI achieves the largest VRFs (also described as the highest RQMC efficiency in some works \cite{zhangEfficientImportanceSampling2021}) for different \(K\). When \(K=60,100\), the VRFs of ODIS based IS\_AS\_PI are slightly smaller than those of AS\_PI, whereas they are larger for \(K=140,180\). 

\begin{table}[!tb]
    \centering
    \footnotesize
    \caption{VRFs and SEs for Asian option pricing with $n=2^{17}$ and $m=50$.}
    \label{tab:AsianPricingVRF-STD-m17-p17}
    \resizebox{\linewidth}{!}{%
    \begin{tabular}{|c|c|c|c|c|c|c|c|c|}
        \hline
        \multirow{2}{*}{$K$}
        & \multirow{2}{*}{MC}
        & \multirow{2}{*}{RQMC}
        & \multirow{2}{*}{PI\_GPCA}
        & \multicolumn{2}{c|}{PI\_IS\_GPCA}
        & \multirow{2}{*}{AS\_PI}
        & \multicolumn{2}{c|}{IS\_AS\_PI} \\
        \cline{5-6}\cline{8-9}
        & & & & ODIS & LapIS & & ODIS & LapIS \\
        \hline
        \multicolumn{9}{|c|}{V R F} \\
        \hline
        60
        & 1 & 2.4e+03 & 4.3e+05 & 1.1e+07 & 4.3e+06
        & 4.8e+07 & 4.0e+07 & 5.6e+08 \\
        100
        & 1 & 5.8e+01 & 2.4e+05 & 9.4e+06 & 1.1e+04
        & 5.5e+07 & 4.9e+07 & 1.4e+09 \\
        140
        & 1 & 9.4e+00 & 3.8e+04 & 1.2e+07 & 2.9e+02
        & 5.2e+07 & 8.8e+07 & 4.2e+09 \\
        180
        & 1 & 1.8e+00 & 3.9e+03 & 3.1e+07 & 8.9e+02
        & 1.5e+06 & 3.9e+08 & 2.4e+10 \\
        \hline
        \multicolumn{9}{|c|}{S E} \\
        \hline
        60
        & 9.0e-03 & 1.8e-04 & 1.4e-05 & 2.7e-06 & 4.3e-06
        & 1.3e-06 & 1.4e-06 & 3.8e-07 \\
        100
        & 6.7e-03 & 8.8e-04 & 1.4e-05 & 2.2e-06 & 6.5e-05
        & 9.0e-07 & 9.6e-07 & 1.8e-07 \\
        140
        & 2.7e-03 & 8.9e-04 & 1.4e-05 & 8.0e-07 & 1.6e-04
        & 3.8e-07 & 2.9e-07 & 4.2e-08 \\
        180
        & 9.0e-04 & 6.7e-04 & 1.4e-05 & 1.6e-07 & 3.0e-05
        & 7.3e-07 & 4.5e-08 & 5.8e-09 \\
        \hline
    \end{tabular}%
    }
\end{table}

In summary, additional LapIS improves the performance in all settings. However, additional ODIS improves the performance only in the out-of-the-money cases.

\subsubsection{Impact of Additional IS on Computational Efficiency}
\label{sec:AsianTime}
In this part, we consider the IS\_AS\_PI and AS\_PI methods. 
We record the runtimes of these methods to analyze the additional computational cost of IS. Here, we define the variance ratio \(R_{\mathrm{var}}\) and the runtime ratio \(R_{\mathrm{time}}\) of our proposed method to AS\_PI by
\[
R_{\mathrm{var}}
=\frac{\text{variance of IS\_AS\_PI}}{\text{variance of AS\_PI}},
\qquad
R_{\mathrm{time}}
=\frac{\text{runtime of IS\_AS\_PI}}{\text{runtime of AS\_PI}}.
\]
We also report
\(R_{\mathrm{CE}}
= R_{\mathrm{var}}\times R_{\mathrm{time}}\)
to compare computational efficiency.
A value of \(R_{\mathrm{CE}}<1\) indicates that
IS\_AS\_PI is more computationally efficient than AS\_PI.

In this experiment, we use $n=2^{14}$ points per replication and perform $m=50$ independent replications. All computations are performed in Python on an Apple M2 computer with 8 CPU cores and 16 GB of memory. 
For each method, the recorded runtime covers the computations
from initialization to the completion of all $m$ pricing
replications, including sample generation.
The results are reported in Table~\ref{tab:AsianComputationalEfficiency}.

\begin{table}[!tb]
    \centering
    \footnotesize
    \caption{\(R_{\mathrm{var}}\), \(R_{\mathrm{time}}\) and \(R_{\mathrm{CE}}\) with \(n=2^{14}\) and \(m=50\).}
    \label{tab:AsianComputationalEfficiency}
    \begin{tabular}{|c|c|c|c|c|c|c|}
        \hline
        \multirow{2}{*}{$K$}
        & \multicolumn{3}{c|}{IS\_AS\_PI (ODIS)}
        & \multicolumn{3}{c|}{IS\_AS\_PI (LapIS)} \\
        \cline{2-7}
        & $R_{\mathrm{var}}$
        & $R_{\mathrm{time}}$
        & $R_{\mathrm{CE}}$
        & $R_{\mathrm{var}}$
        & $R_{\mathrm{time}}$
        & $R_{\mathrm{CE}}$ \\
        \hline
        60  & 118.19\% & 104.03\% & 122.95\%
            & 4.35\% & 108.63\% & 4.72\% \\
        100 & 108.82\% & 107.08\% & 116.53\%
            & 4.56\% & 111.50\% & 5.08\% \\
        140 & 52.18\%  & 112.63\% & 58.76\%
            & 1.31\% & 113.18\% & 1.48\% \\
        180 & 0.29\%   & 103.46\% & 0.29\%
            & 0.01\% & 103.10\% & 0.01\% \\
        \hline
    \end{tabular}%
\end{table}

We focus on \(R_{\mathrm{time}}\) and 
\(R_{\mathrm{CE}}\) reported in
Table~\ref{tab:AsianComputationalEfficiency}.
For \(R_{\mathrm{time}}\), we observe that IS\_AS\_PI has a longer runtime
than AS\_PI, regardless of whether ODIS or LapIS is used.
The values of \(R_{\mathrm{time}}\) vary across the strike prices \(K\) considered.
Furthermore, ODIS based IS\_AS\_PI generally has a shorter
runtime than LapIS based IS\_AS\_PI, except at \(K=180\).
These variations may be partly attributable to the time required
to solve for the preintegration boundaries in 
both IS\_AS\_PI and AS\_PI, 
which affects the runtimes of both methods and hence their ratio.
For \(R_{\mathrm{CE}}\), we observe that
IS\_AS\_PI with ODIS is more computationally efficient than
AS\_PI at \(K=140,180\), as indicated by
\(R_{\mathrm{CE}}<1\), while it is less computationally efficient at \(K=60,100\).
Moreover, LapIS based IS\_AS\_PI is substantially more
computationally efficient than AS\_PI for all strike prices
considered.

In summary, both ODIS based and LapIS based IS\_AS\_PI
have longer runtimes than AS\_PI.
However, LapIS based IS\_AS\_PI is substantially more
computationally efficient than AS\_PI for all cases 
considered, while ODIS based IS\_AS\_PI is more efficient
in some cases.

\subsubsection{Empirical Failure Frequency}
\label{sec:AsianFailure}
In this part, we consider IS\_AS\_PI and AS\_PI methods.
As mentioned in Section~\ref{sec:223ASandPreint},
we find that, for some deep out-of-the-money options, 
all samples used to estimate the gradient information matrix may 
result in a zero matrix, which provides no useful direction for the subsequent preintegration. 
We aim to estimate the probabilities of this empirical failure for AS\_PI and IS\_AS\_PI. For each method, a trial is classified as an empirical failure if the estimated gradient information matrix is a zero matrix.
Since the empirical failure usually occurs in deep out-of-the-money options, we consider \(K=140,180,220,260\). 
We estimate \(C\)  using \(M\in\{64,128,256\}\) RQMC samples,
and \(500\) independently randomized trials are conducted for each combination of \(K\) and \(M\).
Within each trial, the same random seed is used across
methods, strike prices \(K\), and sample sizes \(M\)
for a fair comparison.
The empirical failure frequencies for AS\_PI are reported in Table~\ref{tab:AsianFail}, while the results for both ODIS and LapIS based IS\_AS\_PI are all zero.

\begin{table}[!tb]
    \centering
    \footnotesize
    \caption{Empirical failure frequencies for AS\_PI with \(M=64,128,256\).}
    \label{tab:AsianFail}
    \begin{tabular}{|c|c|c|c|}
        \hline
        $K$ & $M=64$ & $M=128$ & $M=256$ \\
        \hline
        140 & 0.0\%  & 0.0\%  & 0.0\%  \\
        180 & 47.0\% & 21.6\% & 4.4\%  \\
        220 & 94.8\% & 88.2\% & 76.0\% \\
        260 & 99.8\% & 98.4\% & 96.6\% \\
        \hline
    \end{tabular}
\end{table}

From Table~\ref{tab:AsianFail}, the empirical failure frequency for AS\_PI increases with the strike price \(K\) for each \(M\). 
It generally decreases as \(M\) increases, but remains high for large \(K\). 
In particular, at \(K=220\) and \(260\), the empirical
failure frequencies remain at \(76.0\%\) and \(96.6\%\),
respectively, even with \(M=256\).
For the run used to produce Figure~\ref{fig:K180},
the estimated gradient information matrix is nonzero.
The empirical failure frequencies are reported separately. 
In contrast, for IS\_AS\_PI, no empirical failures are observed. 
Therefore, IS\_AS\_PI successfully overcomes the problem of the empirical failure in some deep out-of-the-money option pricing problems.

\subsection{Sensitivity Analysis: Estimating the Delta of Asian Options}
\label{sec:AsianDeltaConvergence}
We use IS\_AS\_PI to estimate the Delta of Asian options. 
As mentioned in Section~\ref{sec:DeltaOfAsianOption}, 
we use the substitution in \cite{wangHandlingDiscontinuitiesFinancial2016}
to cope with the jump in the integrand of the PW estimator before using IS\_AS\_PI. 
Thus IS\_AS\_PI is actually applied to the transformed integrand \eqref{eq:c4DeltaSmoothIntegrand}.
The model parameters and strike prices are the same as those in Section~\ref{sec:AsianExperiment}.
For each sample size $n=2^N$ with $N=5,6,\dots,16$, we perform $m=50$ independent replications to estimate the SEs for methods based on RQMC. 
The numerical results are shown in Figure~\ref{fig:delta-convergence}. The axes, markers, dashed lines and annotations in the four panels have the same meanings as those described in Section~\ref{sec:AsianConvergence}. 
For AS\_PI and IS\_AS\_PI, the prefix ``S'' in the legend indicates that the substitution is performed before applying the method. 
This prefix is omitted in the text for brevity.
For PI\_GPCA and PI\_IS\_GPCA, the initial
preintegration step removes the jump in the PW Delta
integrand \cite{zhangEfficientImportanceSampling2021}.
Thus these two methods are applied without
the preceding substitution.
Here, we focus on comparing IS\_AS\_PI and AS\_PI.
Since the same substitution is applied to both methods,
it does not destroy the fairness of this comparison.

\begin{figure}[!b]
    \centering

    \subfloat[$K=60$]{%
        \includegraphics[width=0.48\linewidth]{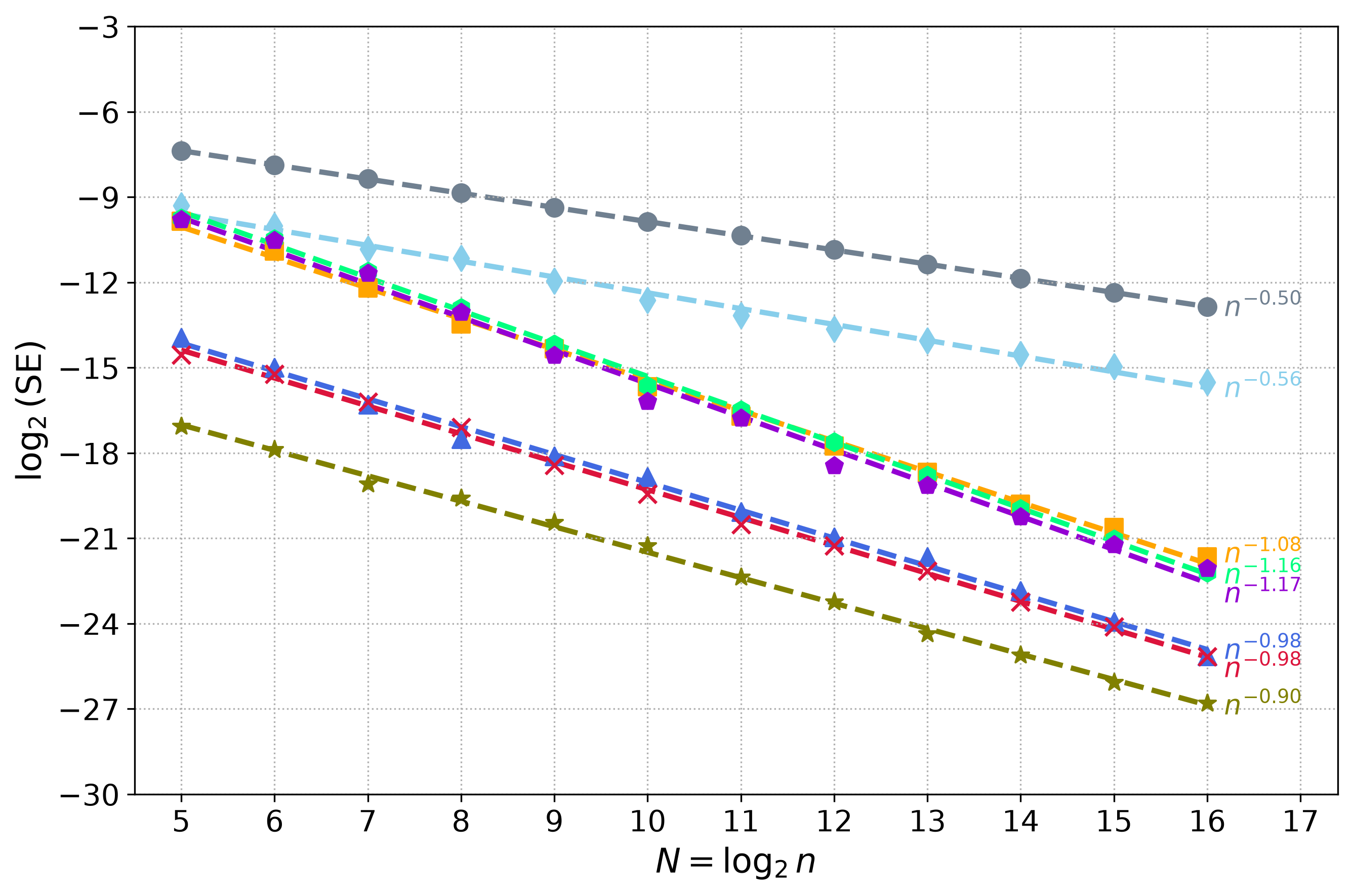}%
        \label{fig:DeltaK60}%
    }
    \hfill
    \subfloat[$K=100$]{%
        \includegraphics[width=0.48\linewidth]{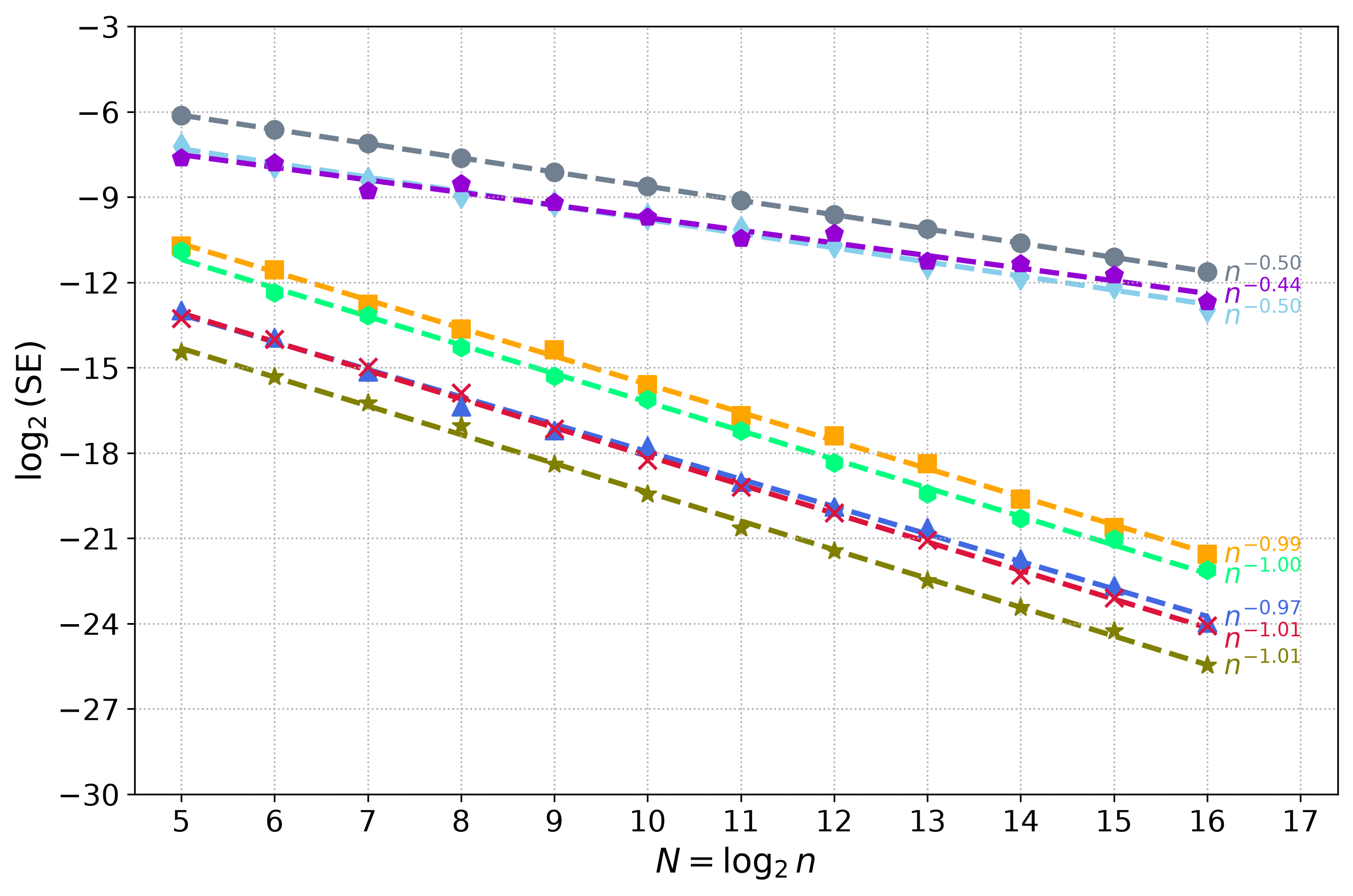}%
        \label{fig:DeltaK100}%
    }

    \subfloat[$K=140$]{%
        \includegraphics[width=0.48\linewidth]{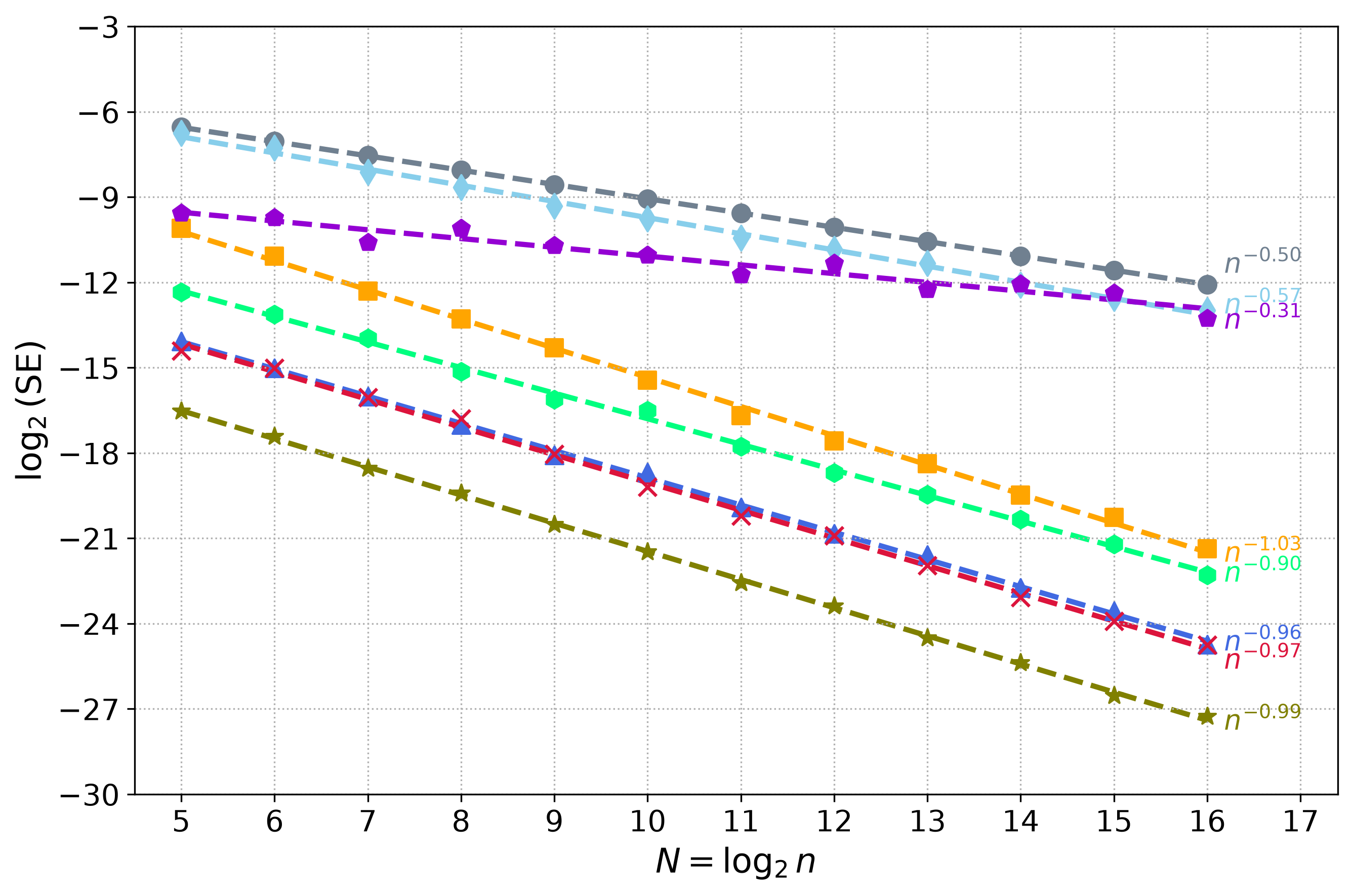}%
        \label{fig:DeltaK140}%
    }
    \hfill
    \subfloat[$K=180$]{%
        \includegraphics[width=0.48\linewidth]{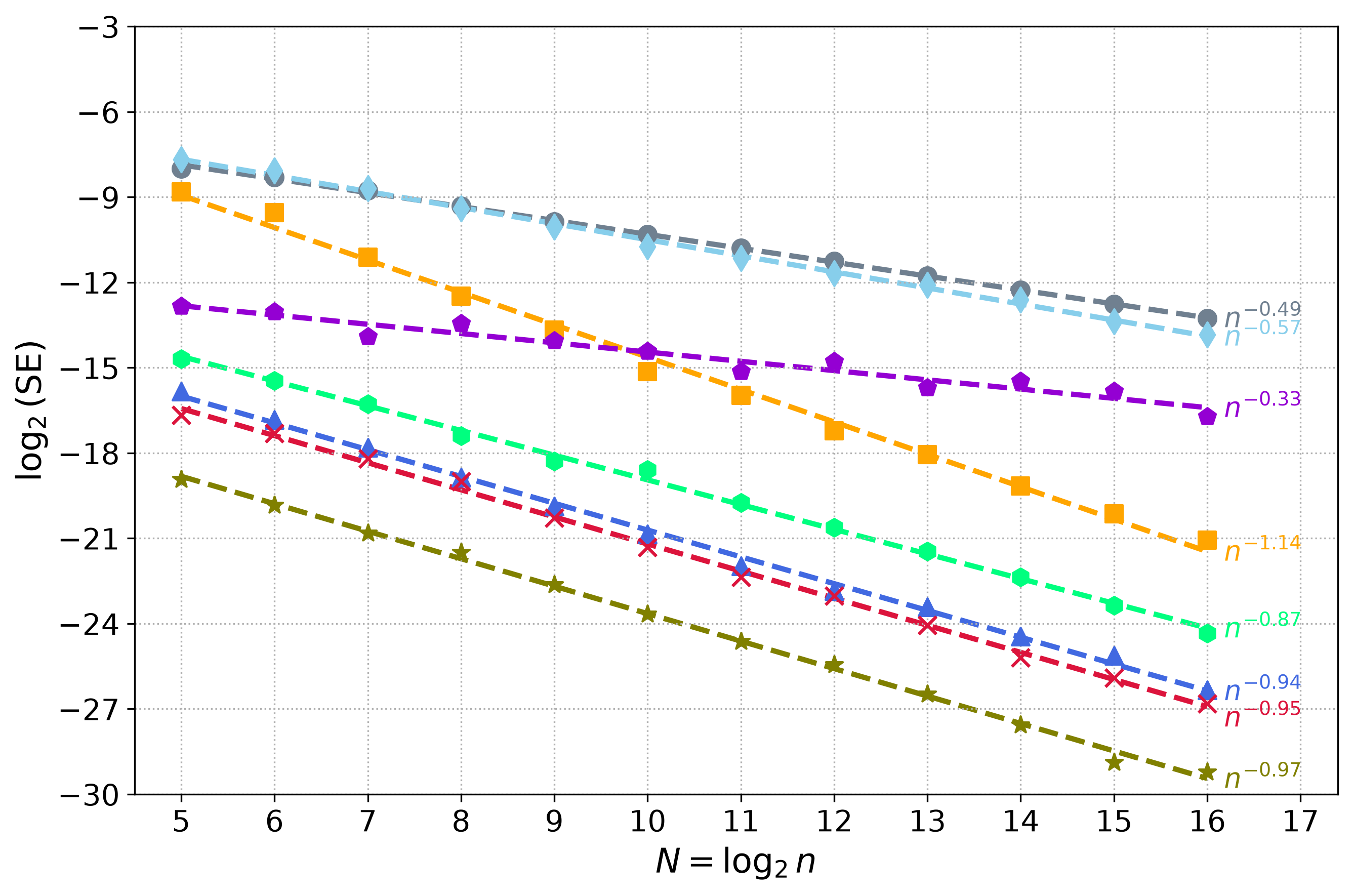}%
        \label{fig:DeltaK180}%
    }

    \par\smallskip

    \includegraphics[width=0.7\linewidth]{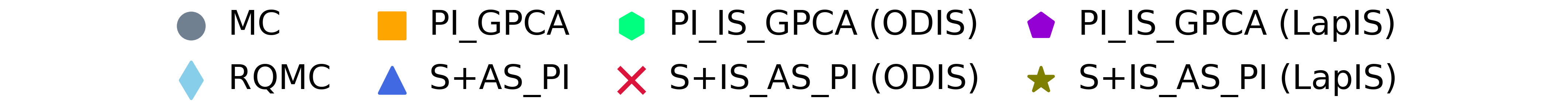}

    \caption{Convergence behavior for estimating the Delta with different strike prices.}
    \label{fig:delta-convergence}
\end{figure}

\begin{table}[!t]
    \centering
    \footnotesize
    \caption{VRFs and SEs for estimating the Delta with $n=2^{16}$ and $m=50$.}
    \label{tab:AsianDeltaVRF}
    \resizebox{\linewidth}{!}{%
    \begin{tabular}{|c|c|c|c|c|c|c|c|c|}
        \hline
        \multirow{2}{*}{$K$}
        & \multirow{2}{*}{MC}
        & \multirow{2}{*}{RQMC}
        & \multirow{2}{*}{PI\_GPCA}
        & \multicolumn{2}{c|}{PI\_IS\_GPCA}
        & \multirow{2}{*}{S+AS\_PI}
        & \multicolumn{2}{c|}{S+IS\_AS\_PI} \\
        \cline{5-6}\cline{8-9}
        & & & & ODIS & LapIS & & ODIS & LapIS \\
        \hline
        \multicolumn{9}{|c|}{V R F} \\
        \hline
        60
        & 1 & 3.9e+01 & 2.0e+05 & 4.3e+05 & 3.5e+05
        & 2.5e+07 & 2.6e+07 & 2.5e+08 \\
        100
        & 1 & 6.3e+00 & 9.6e+05 & 2.1e+06 & 4.4e+00
        & 2.7e+07 & 3.1e+07 & 2.2e+08 \\
        140
        & 1 & 3.6e+00 & 3.9e+05 & 1.4e+06 & 5.3e+00
        & 4.2e+07 & 4.3e+07 & 1.4e+09 \\
        180
        & 1 & 2.2e+00 & 4.9e+04 & 4.6e+06 & 1.2e+02
        & 7.4e+07 & 1.4e+08 & 4.0e+09 \\
        \hline
        \multicolumn{9}{|c|}{S E} \\
        \hline
        60
        & 1.3e-04 & 2.1e-05 & 3.0e-07 & 2.1e-07 & 2.3e-07
        & 2.7e-08 & 2.7e-08 & 8.6e-09 \\
        100
        & 3.2e-04 & 1.3e-04 & 3.2e-07 & 2.2e-07 & 1.5e-04
        & 6.1e-08 & 5.7e-08 & 2.1e-08 \\
        140
        & 2.3e-04 & 1.2e-04 & 3.7e-07 & 2.0e-07 & 1.0e-04
        & 3.6e-08 & 3.5e-08 & 6.2e-09 \\
        180
        & 1.0e-04 & 6.9e-05 & 4.5e-07 & 4.7e-08 & 9.3e-06
        & 1.2e-08 & 8.5e-09 & 1.6e-09 \\
        \hline
    \end{tabular}%
    }
\end{table}

From Figure~\ref{fig:delta-convergence}, we observe that, at the same sample sizes, LapIS based IS\_AS\_PI achieves the smallest estimated SEs among the methods compared in all settings. 
ODIS based IS\_AS\_PI generally achieves the
second smallest estimated SEs,
 although AS\_PI yields
smaller estimated SEs at some smaller sample sizes for $K=60,100,140$.
The advantage of ODIS based IS\_AS\_PI is more pronounced in the case of $K=180$.

Table~\ref{tab:AsianDeltaVRF} reports the VRFs
and SEs of the methods for $n=2^{16}$,
estimated using $m=50$ independent replications.
We focus on the VRFs for simplicity.
In this table, 
LapIS based IS\_AS\_PI achieves the largest VRFs
for all four strike prices, while ODIS based
IS\_AS\_PI achieves the second largest VRFs.
Therefore, at $n=2^{16}$, additional IS improves
the estimated variance reduction performance of
AS\_PI in all four cases, with substantially
larger improvements obtained using LapIS.

The results show that the advantage of the proposed IS\_AS\_PI method can also be seen in the Delta estimation. 
For Delta estimation, no empirical failures are observed in
AS\_PI or IS\_AS\_PI for the deep out-of-the-money
strike prices considered in Section~\ref{sec:AsianFailure}.
The reason is that the substitution removes the flat zero regions caused by the indicator function, making it less likely that all sampled gradients are zero and hence that the estimate of $C$ is a zero matrix.

\subsection{Basket Option Pricing}
Experiments in this section aim to evaluate the performance of IS\_AS\_PI in basket option pricing.
The following parameters are chosen as in \cite{liuPreintegrationActiveSubspace2023}: $d=50$ time steps (giving a problem dimension of \(2d\)), maturity $T=1$, correlation coefficient \(\rho=0.5\), risk-free interest rate $r=0.1$, volatilities $\sigma_1=0.1,\sigma_2=0.2$, initial stock prices $S_0^{(1)}=S_0^{(2)}=100$ and weights \(w_1=0.8, w_2 = 0.2\). We consider strike prices $K=80,100,120$ in Section~\ref{sec:BasketConvergence}.

\subsubsection{Convergence Analysis}
\label{sec:BasketConvergence}
This experiment compares the convergence behavior of the methods mentioned at the beginning of Section~\ref{sec:Experiment}.
For each sample size $n=2^N$ with $N=5,6,\dots,16$, we perform $m=50$ independent replications to estimate the SEs for methods based on RQMC.
The numerical results are shown in Figure~\ref{fig:basket-convergence}. The axes, markers, dashed lines and annotations in the three panels have the same meanings as in Section~\ref{sec:AsianConvergence}. For a sample size of $n=2^{16}$, with $m=50$ independent replications, VRFs and SEs for the methods are reported in Table~\ref{tab:BasketPricingVRF}. 

\begin{figure}[!tb]
    \centering

    \subfloat[$K=80$]{%
        \includegraphics[width=0.33\linewidth]{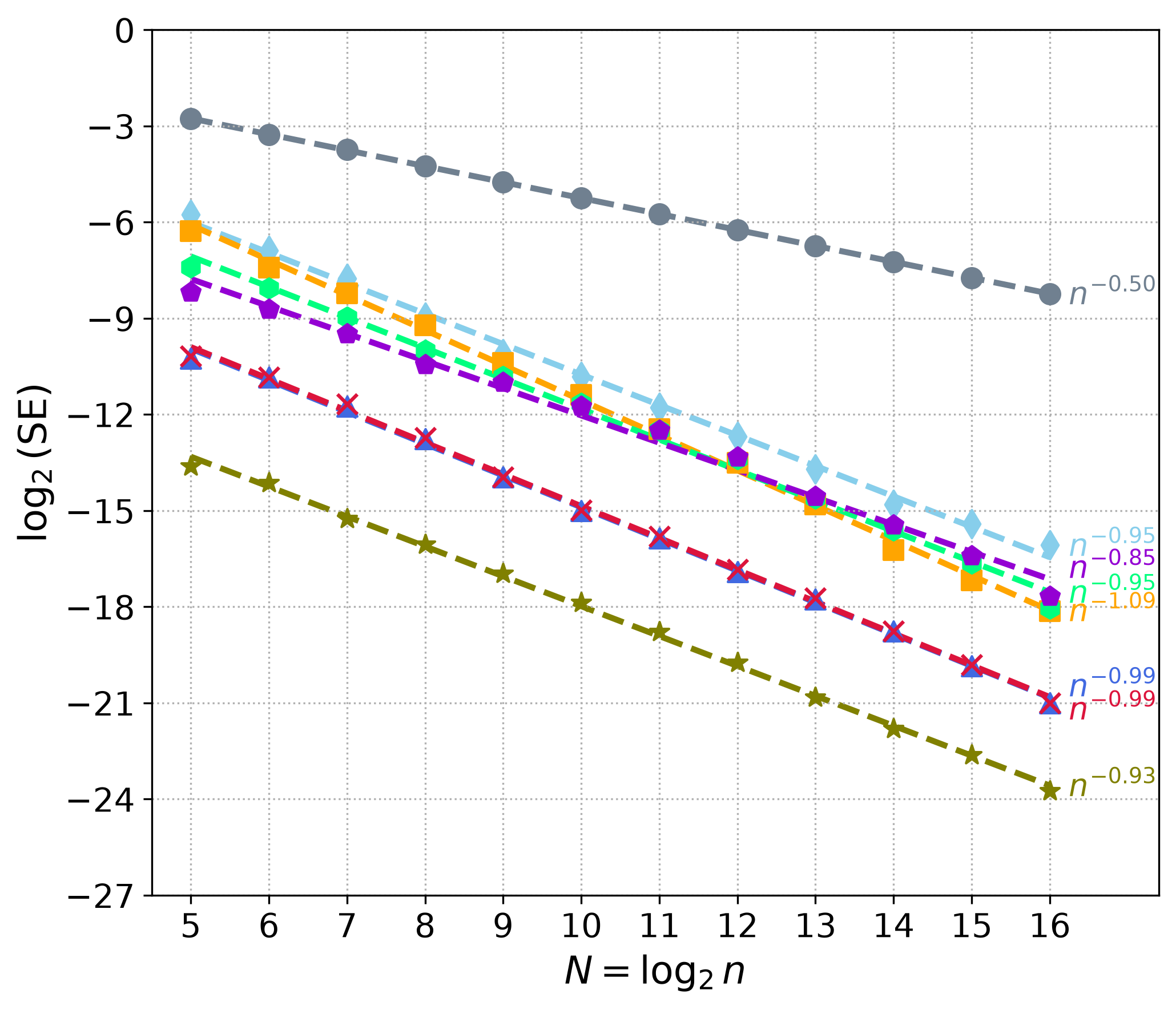}%
        \label{fig:BasketK80}%
    }
    \hfill
    \subfloat[$K=100$]{%
        \includegraphics[width=0.33\linewidth]{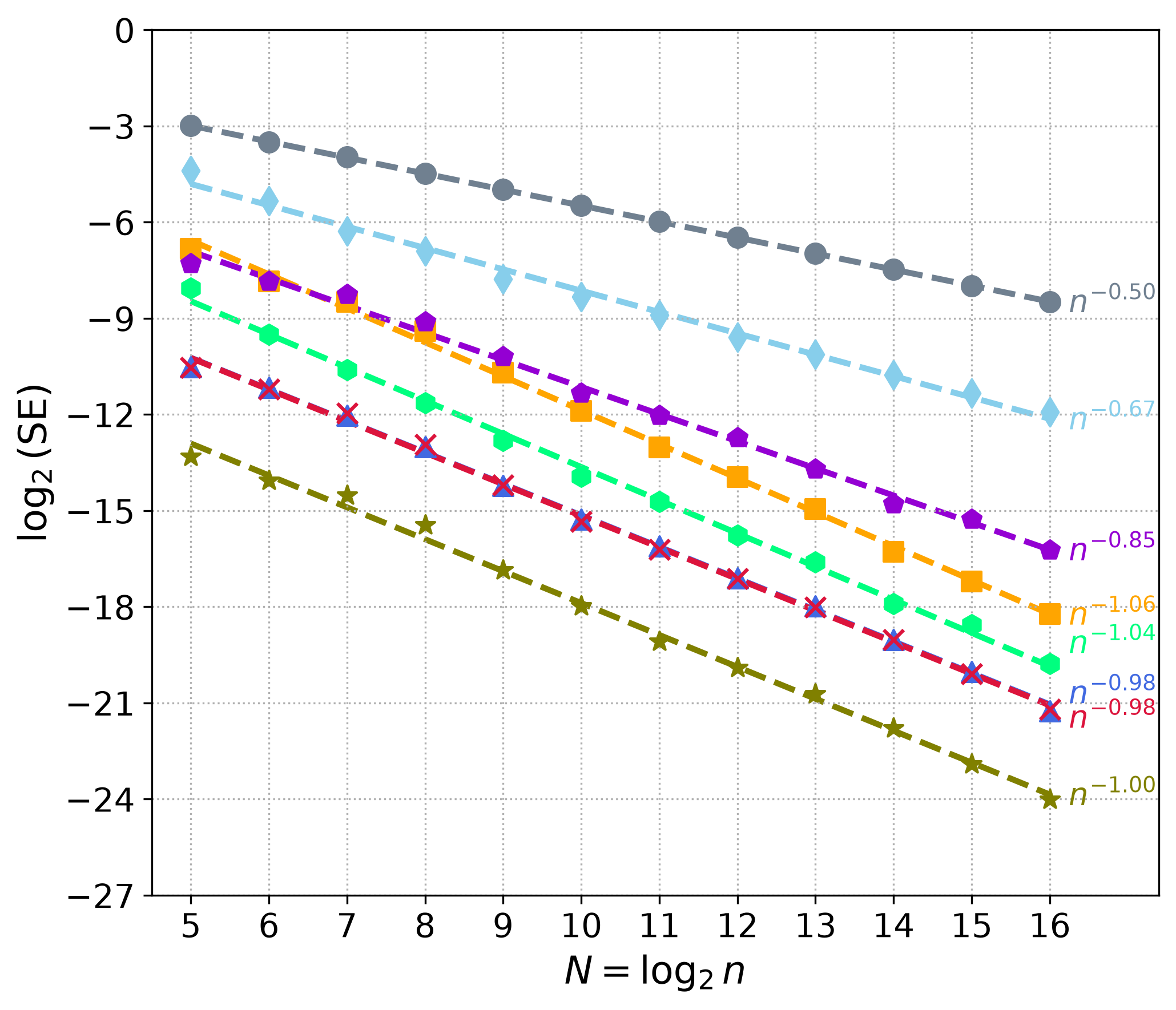}%
        \label{fig:BasketK100}%
    }
    \hfill
    \subfloat[$K=120$]{%
        \includegraphics[width=0.33\linewidth]{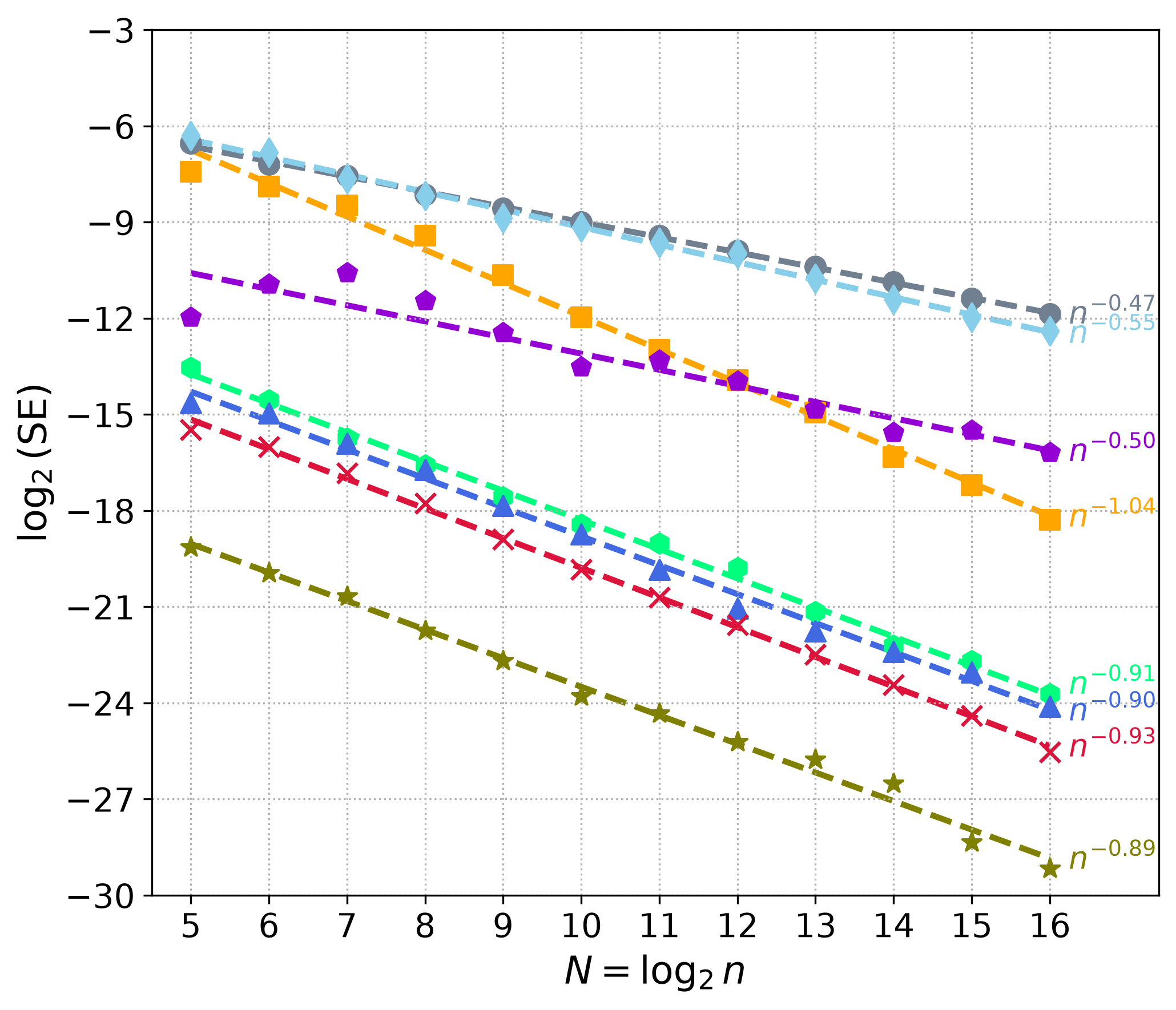}%
        \label{fig:BasketK120}%
    }

    \par\smallskip

    \includegraphics[width=0.7\linewidth]{figures/Asian-shared-legend.png}

    \caption{Convergence behavior for basket option pricing with different strike prices.}
    \label{fig:basket-convergence}
\end{figure}

\begin{table}[!tb]
    \centering
    \footnotesize
    \caption{VRFs and SEs for basket option pricing with $n=2^{16}$ and $m=50$.}
    \label{tab:BasketPricingVRF}
    \resizebox{\linewidth}{!}{%
    \begin{tabular}{|c|c|c|c|c|c|c|c|c|}
        \hline
        \multirow{2}{*}{$K$}
        & \multirow{2}{*}{MC}
        & \multirow{2}{*}{RQMC}
        & \multirow{2}{*}{PI\_GPCA}
        & \multicolumn{2}{c|}{PI\_IS\_GPCA}
        & \multirow{2}{*}{AS\_PI}
        & \multicolumn{2}{c|}{IS\_AS\_PI} \\
        \cline{5-6}\cline{8-9}
        & & & & ODIS & LapIS & & ODIS & LapIS \\
        \hline
        \multicolumn{9}{|c|}{V R F} \\
        \hline
        80
        & 1 & 5.2e+04 & 9.1e+05 & 8.0e+05 & 4.8e+05
        & 4.9e+07 & 4.8e+07 & 2.1e+09 \\
        100
        & 1 & 1.2e+02 & 7.4e+05 & 6.2e+06 & 4.6e+04
        & 4.9e+07 & 4.6e+07 & 2.2e+09 \\
        120
        & 1 & 2.1e+00 & 7.3e+03 & 1.4e+07 & 4.0e+02
        & 2.4e+07 & 1.7e+08 & 2.6e+10 \\
        \hline
        \multicolumn{9}{|c|}{S E} \\
        \hline
        80
        & 3.3e-03 & 1.5e-05 & 3.5e-06 & 3.7e-06 & 4.8e-06
        & 4.7e-07 & 4.8e-07 & 7.2e-08 \\
        100
        & 2.8e-03 & 2.6e-04 & 3.3e-06 & 1.1e-06 & 1.3e-05
        & 4.0e-07 & 4.1e-07 & 5.9e-08 \\
        120
        & 2.7e-04 & 1.8e-04 & 3.1e-06 & 7.3e-08 & 1.3e-05
        & 5.5e-08 & 2.0e-08 & 1.7e-09 \\
        \hline
    \end{tabular}%
    }
\end{table}

From Figure~\ref{fig:basket-convergence} and Table~\ref{tab:BasketPricingVRF}, we observe LapIS based IS\_AS\_PI performs best among methods compared in all settings considered, because it achieves the smallest estimated SEs in the figure and has the largest VRFs in the table. 
The performance of ODIS based IS\_AS\_PI is close to the performance of AS\_PI for \(K=80\) and \(K=100\), as shown in Figures~\ref{fig:BasketK80} and \ref{fig:BasketK100}. However, ODIS based IS\_AS\_PI performs better than AS\_PI for \(K=120\) in Figure~\ref{fig:BasketK120}, indicated by the lower SEs. 
The VRFs in the table support similar conclusions.

In summary, additional LapIS improves the performance in all settings, while additional ODIS improves the performance only in the out-of-the-money case.

\subsubsection{Empirical Failure Frequency}
\label{sec:BasketFailure}
We also evaluate the empirical failure frequency as described in Section~\ref{sec:AsianFailure}. We consider IS\_AS\_PI and AS\_PI methods and strike prices \(K=110,120,130,140\). 
For each combination of strike price \(K\) and sample size \(M\in\{64,128,256\}\), we conduct $500$ independently randomized trials.
The seed strategy is the same as that used in Section~\ref{sec:AsianFailure} to ensure a fair comparison. 
The empirical failure frequencies for AS\_PI are reported in Table~\ref{tab:BasketFail}, and the corresponding results for both ODIS and LapIS based IS\_AS\_PI are all zero.

\begin{table}[!tb]
    \centering
    \footnotesize
    \caption{Empirical failure frequencies for AS\_PI with \(M=64,128,256\).}
    \label{tab:BasketFail}
    \begin{tabular}{|c|c|c|c|}
        \hline
        $K$ & $M=64$ & $M=128$ & $M=256$ \\
        \hline
        110 & 0.0\%   & 0.0\%   & 0.0\%   \\
        120 & 24.8\%  & 6.6\%   & 0.4\%   \\
        130 & 98.4\%  & 95.2\%  & 88.8\%  \\
        140 & 100.0\% & 100.0\% & 100.0\% \\
        \hline
    \end{tabular}
\end{table}

The results in Table~\ref{tab:BasketFail} have similar trends to those reported in Section~\ref{sec:AsianFailure}. 
The AS\_PI convergence curve for \(K=120\) shown in Figure~\ref{fig:BasketK120} means a nonzero estimate of \(C\) is obtained in the trial.
For \(K=140\), the empirical failure frequency
for AS\_PI is \(100\%\) across all three values of \(M\),
whereas no failures are observed for either ODIS based or
LapIS based IS\_AS\_PI.
Therefore, IS\_AS\_PI successfully avoids the empirical failure in some deep out-of-the-money option pricing problems. 

\section{Conclusion}
\label{sec:conclusion}

This work proposed a three-step method that combines IS, the AS method and preintegration, referred to as the IS\_AS\_PI method. The method is designed to leverage the complementary strengths of these techniques within the framework of QMC methods.
Theoretical analysis established the orthogonal invariance of both the active subspace and IS-active subspace, providing a solid foundation for the proposed method. 
Numerical experiments show that the proposed method performs well in variance reduction for option pricing problems and Delta estimation problems considered, indicating that it provides a useful tool for evaluating high-dimensional integrals in computational finance.

Future work will explore the potential application of the IS\_AS\_PI
method to high-dimensional integration problems in other fields.
Another direction is to investigate the reasons for the high efficiency of the proposed method, particularly
when Laplace IS is used.

\appendix
\section{Proofs in Section~\ref{sec:ISAS}} 
\subsection{Proof of Theorem~\ref{thm:c3ivarrianceofAS}}
\label{prf:thm:c3ivarrianceofAS}
\begin{proof}
  Let \(\bm Z\sim N(\bm 0,I_d)\).

  (1) The chain rule gives \(\nabla g_2(\bm{y}) = U^{\mathsf{T}}\nabla g_1(U\bm{y})\) almost everywhere.
  Since \(U\bm{Z}\) and \(\bm{Z}\) have the same distribution, we have
  \[
      C_2 = \mathbb{E}\!\left(\nabla g_2(\bm{Z})\nabla g_2(\bm{Z})^{\mathsf{T}}\right)
        = U^{\mathsf{T}} \mathbb{E}\!\left(\nabla g_1(U\bm{Z})\nabla g_1(U\bm{Z})^{\mathsf{T}}\right) U 
        = U^{\mathsf{T}} C_1 U .
  \]
  Thus $C_1$ and $C_2$ have the same eigenvalues $\lambda_1\ge\cdots\ge\lambda_d$. 
  Since \(Q= U^{\mathsf{T}} P\) is orthogonal, \(C_2\) has the eigendecomposition
  \(
    Q\Lambda Q^{\mathsf{T}}=U^{\mathsf{T}} P\Lambda P^{\mathsf{T}} U = U^{\mathsf{T}} C_1 U=C_2.
  \)

  (2) For every \(1\le r\le d\), the \(r\)-dimensional active subspaces of $g_1(\bm{x})$ and $g_2(\bm{y})$ are \(\operatorname{span}\{P_{[:,1]},\ldots,P_{[:,r]}\}\) and \(\operatorname{span}\{Q_{[:,1]},\ldots,Q_{[:,r]}\}\) in the \(\bm x\)- and \(\bm y\)-coordinate systems, respectively.
  Since \(P=UQ\), their representations
  in the \(\bm x\)-coordinate system satisfy
  \[
  \operatorname{span}\{P_{[:,1]},\dots,P_{[:,r]}\}
  =
  U\operatorname{span}\{Q_{[:,1]},\dots,Q_{[:,r]}\},
  \]
  which means their \(r\)-dimensional active subspaces are the same. 


  (3) The relation \(Q= U^{\mathsf{T}} P\) yields
  \[
    g_2^\mathrm{A}(\bm{z}) 
    = g_2(Q\bm{z}) 
    = g_1(UQ\bm{z})
    = g_1(UU^{\mathsf{T}} P\bm{z}) 
    = g_1(P\bm{z}) 
    = g_1^\mathrm{A}(\bm{z}).
  \]
\end{proof}


\subsection{Proof of Lemma~\ref{lem:c3lemma1}}
\label{prf:lem:c3lemma1}
\begin{proof}
  Since $R_1R_1^{\mathsf{T}}=\Sigma$, we have $\det(R_1)\neq 0$, and thus $R_1$ is invertible. Let $U=R_1^{-1}R_2$, and the conclusion can be verified directly.
\end{proof}

\subsection{Proof of Theorem~\ref{thm:c3ivarrianceofISAS}}
\label{prf:thm:c3ivarrianceofISAS}
\begin{proof}
  It suffices to prove that, for each of the two IS methods,
  there exists an orthogonal matrix \(V\) such that
  \(
  g_2^\mathrm{I}(\bm y)=g_1^\mathrm{I}(V\bm y).
  \)
  Conclusions (1)--(3) then follow from
  Theorem~\ref{thm:c3ivarrianceofAS}.
    
  First, we prove it for optimal drift IS. Define
\(D_{1,+}=\{\bm x:g_1(\bm x)>0\}\) and
\(D_{2,+}=\{\bm y:g_2(\bm y)>0\}\).
Since \(g_2(\bm y)=g_1(U\bm y)\),
we have \(D_{1,+}=UD_{2,+}\). Let $\bm{\mu}_1$ and $\bm{\mu}_2$ denote the optimal drifts obtained by solving  equation~\eqref{eq:c3ISODISmiu} for $g_1(\bm{x})$ and $g_2(\bm{y})$, respectively. Then
  \[
    \frac{\nabla g_1(\bm{\mu}_1)}{g_1(\bm{\mu}_1)} = \bm{\mu}_1,
    \qquad
    \frac{U^{\mathsf{T}} \nabla g_1(U \bm{\mu}_2)}{g_1(U \bm{\mu}_2)} = \bm{\mu}_2.
  \]
  Since \(U\bm\mu_2\in D_{1,+}\), by the uniqueness of the solution to~\eqref{eq:c3ISODISmiu}, it follows that \(\bm{\mu}_1 = U\bm{\mu}_2.\)
  According to~\eqref{eq:c3ISODISintegrand}, we have
  \begin{align}
    g_2^\mathrm{I}(\bm{y})
    &= g_2(\bm{y}+\bm{\mu}_2)\exp\!\left(-\bm{\mu}_2^{\mathsf{T}}\bm{y}-\tfrac{1}{2}\bm{\mu}_2^{\mathsf{T}}\bm{\mu}_2\right) \notag\\
    &= g_1(U\bm{y}+\bm{\mu}_1)\exp\!\left(-\bm{\mu}_1^{\mathsf{T}}U\bm{y}-\tfrac{1}{2}\bm{\mu}_1^{\mathsf{T}}\bm{\mu}_1\right) 
    = g_1^\mathrm{I}(U\bm{y}).
  \end{align}
  Thus there exists an orthogonal matrix \(V=U\) such that $g_2^\mathrm{I}(\bm{y})=g_1^\mathrm{I}(V\bm{y})$.

  Second, we prove it for Laplace IS. Let $\bm{\mu}_1,\,\bm{\mu}_2$ and $\Gamma_1,\,\Gamma_2$ denote the drift vectors and covariance matrices obtained by Laplace IS for $g_1(\bm{x})$ and $g_2(\bm{y})$, respectively. As discussed above, we already know that \(\bm{\mu}_1 = U\bm{\mu}_2.\)
  Define $F_1(\bm{x})=\ln(g_1(\bm{x}))$ on \(D_{1,+}\) and $F_2(\bm{y})=\ln(g_2(\bm{y}))$ on \(D_{2,+}\). Then \(F_2(\bm{y}) = F_1(U\bm{y})\), 
  which yields
  \(
    \nabla^2F_2(\bm{y}) = U^{\mathsf{T}}\nabla^2F_1(U\bm{y})U,
  \) for \(\bm y\) in a neighborhood of \(\bm\mu_2\).
  Using~\eqref{eq:c3ISLapISGamma}, we obtain
  \(
    \Gamma_2 = U^{\mathsf{T}}\Gamma_1 U.
  \)
  Consequently,
  \(
    \det(\Gamma_1)=\det(\Gamma_2).
  \)
  Let \(L_1\) and \(L_2\) be the Cholesky factors of \(\Gamma_1\) and \(\Gamma_2\), respectively, so that \(L_1L_1^{\mathsf{T}}=\Gamma_1\) and \( L_2L_2^{\mathsf{T}}=\Gamma_2\).
  Since
  \(
    \Gamma_2 = U^{\mathsf{T}}L_1L_1^{\mathsf{T}}U,
  \)
  Lemma~\ref{lem:c3lemma1} guarantees the existence of an orthogonal matrix $V$ such that
  \(
    L_2 = U^{\mathsf{T}} L_1 V.
  \)
  Using~\eqref{eq:c3ISLapISintegrand}, the integrand after Laplace IS satisfies
  \begin{align}
    g_2^\mathrm{I}(\bm{y})
    &= g_2(L_2\bm{y}+\bm{\mu}_2)\sqrt{\det(\Gamma_2)}
       \exp\!\left(\tfrac{1}{2}\bm{y}^{\mathsf{T}}\bm{y} - \tfrac{1}{2}(L_2\bm{y}+\bm{\mu}_2)^{\mathsf{T}}(L_2\bm{y}+\bm{\mu}_2)\right) \notag\\
    &= g_1(L_1V\bm{y}+\bm{\mu}_1)\sqrt{\det(\Gamma_1)}
       \exp\!\left(\tfrac{1}{2}\bm{y}^{\mathsf{T}}\bm{y} - \tfrac{1}{2}(L_1V\bm{y}+\bm{\mu}_1)^{\mathsf{T}}(L_1V\bm{y}+\bm{\mu}_1)\right) \notag\\
    &= g_1^\mathrm{I}(V\bm{y}).
  \end{align}
  Thus there exists an orthogonal matrix \(V=L_1^{-1}UL_2 \) such that $g_2^\mathrm{I}(\bm{y})=g_1^\mathrm{I}(V\bm{y})$.
\end{proof}

\section*{Acknowledgments}
The authors gratefully thank two anonymous reviewers, whose comments helped us improve this paper greatly.

\bibliographystyle{siamplain} 
\bibliography{references}
\end{document}

%% file: ex_shared.tex
\usepackage{lipsum}
\usepackage{amsfonts}
\usepackage{graphicx}
\usepackage{epstopdf}
\usepackage{algorithmic}
\usepackage{bm}
\ifpdf
  \DeclareGraphicsExtensions{.eps,.pdf,.png,.jpg}
\else
  \DeclareGraphicsExtensions{.eps}
\fi
\usepackage{graphicx,epstopdf} 
\usepackage[caption=false]{subfig} 

\newsiamremark{remark}{Remark}
\newsiamremark{hypothesis}{Hypothesis}
\crefname{hypothesis}{Hypothesis}{Hypotheses}
\newsiamthm{claim}{Claim}
\newsiamremark{fact}{Fact}
\crefname{fact}{Fact}{Facts}

\headers{importance sampling and active subspace in quasi-Monte Carlo}{Jiaxin Yu, Xiaoqun Wang}

\title{importance sampling and active subspace in quasi-Monte Carlo\thanks{Submitted to the editors on March 2, 2026.
}}

\author{Jiaxin Yu\thanks{jasenyu(at)foxmail(dot)com}
\and Xiaoqun Wang}

\usepackage{amsopn}


%% file: references.bib
@book{constantineActiveSubspacesEmerging2015,
  title      = {Active {{Subspaces}}: {{Emerging Ideas}} for {{Dimension Reduction}} in {{Parameter Studies}}},
  author     = {Constantine, Paul G.},
  year       = {2015},
  publisher  = {{Society for Industrial and Applied Mathematics}},
  address    = {Philadelphia, PA},
  isbn       = {978-1-61197-385-3 978-1-61197-386-0},
  langid     = {english},
  month      = mar,
  shorttitle = {Active {{Subspaces}}},
  urldate    = {2024-03-14}
}

@book{glassermanMonteCarloMethods2003,
  title     = {{{Monte Carlo Methods}} in {{Financial Engineering}}},
  author    = {Glasserman, Paul},
  year      = {2003},
  publisher = {Springer},
  address   = {New York, NY},
  editor    = {Rozovskii, B. and Yor, M.},
  isbn      = {978-1-4419-1822-2 978-0-387-21617-1},
  urldate   = {2023-12-21}
}

@book{niederreiterRandomNumberGeneration1992,
  title     = {Random {{Number Generation}} and {{Quasi-Monte Carlo Methods}}},
  author    = {Niederreiter, Harald},
  year      = {1992},
  publisher = {{Society for Industrial and Applied Mathematics}},
  isbn      = {978-0-89871-295-7 978-1-61197-008-1},
  langid    = {english},
  urldate   = {2023-12-21}
}

@article{huDiscoveringOnedimensionalActive2016,
  title    = {Discovering a One-Dimensional Active Subspace to Quantify Multidisciplinary Uncertainty in Satellite System Design},
  author   = {Hu, Xingzhi and Parks, Geoffrey T. and Chen, Xiaoqian and Seshadri, Pranay},
  journal  = {Advances in Space Research},
  year     = {2016},
  volume   = {57},
  number   = {5},
  pages    = {1268--1279},
  issn     = {0273-1177},
  month    = mar,
  urldate  = {2024-04-26}
}

@article{heDimensionReductionSmoothing2017,
  title   = {An Integrated {Quasi-Monte Carlo} Method for Handling High Dimensional Problems with Discontinuities in Financial Engineering},
  author  = {He, Zhijian and Wang, Xiaoqun},
  journal = {Computational Economics},
  year    = {2021},
  volume  = {57},
  number  = {2},
  pages   = {693--718}
}

@article{xiaoEnhancingQuasiMonteCarlo2019,
  title   = {Enhancing {{quasi-Monte Carlo simulation}} by {{minimizing effective dimension}} for {{derivative pricing}}},
  author  = {Xiao, Ye and Wang, Xiaoqun},
  journal = {Computational Economics},
  year    = {2019},
  volume  = {54},
  number  = {1},
  pages   = {343--366},
  issn    = {0927-7099, 1572-9974},
  langid  = {english},
  month   = jun,
  urldate = {2023-12-23}
}

@article{jeffersonActiveSubspacesSensitivity2015,
  title    = {Active Subspaces for Sensitivity Analysis and Dimension Reduction of an Integrated Hydrologic Model},
  author   = {Jefferson, Jennifer and Gilbert, James and Constantine, Paul and Maxwell, Reed},
  journal  = {Computers and Geosciences},
  year     = {2015},
  volume   = {83},
  pages    = {127--138},
  month    = oct
}

@article{wangHowPathGeneration2012,
  title    = {How Do Path Generation Methods Affect the Accuracy of Quasi-{{Monte Carlo}} Methods for Problems in Finance?},
  author   = {Wang, Xiaoqun and Tan, Ken Seng},
  journal  = {Journal of Complexity},
  year     = {2012},
  volume   = {28},
  number   = {2},
  pages    = {250--277},
  issn     = {0885-064X},
  month    = apr,
  urldate  = {2023-12-21}
}

@article{papageorgiouBrownianBridgeDoes2002,
  title    = {The {{Brownian bridge does not offer}} a {{consistent advantage}} in {{quasi-Monte Carlo integration}}},
  author   = {Papageorgiou, A.},
  journal  = {Journal of Complexity},
  year     = {2002},
  volume   = {18},
  number   = {1},
  pages    = {171--186},
  urldate  = {2023-12-25}
}

@article{wangEffectiveDimensionQuasiMonte2003,
  title    = {The Effective Dimension and Quasi-{{Monte Carlo}} Integration},
  author   = {Wang, Xiaoqun and Fang, Kai-Tai},
  journal  = {Journal of Complexity},
  year     = {2003},
  volume   = {19},
  number   = {2},
  pages    = {101--124},
  issn     = {0885-064X},
  month    = apr,
  urldate  = {2023-12-25}
}

@article{xiaoConditionalQuasiMonteCarlo2018,
  title    = {Conditional Quasi-{{Monte Carlo}} Methods and Dimension Reduction for Option Pricing and Hedging with Discontinuous Functions},
  author   = {Xiao, Ye and Wang, Xiaoqun},
  journal  = {Journal of Computational and Applied Mathematics},
  year     = {2018},
  volume   = {343},
  pages    = {289--308},
  issn     = {0377-0427},
  month    = dec,
  urldate  = {2023-12-17}
}

@article{griewankHighDimensionalIntegration2018,
  title   = {High Dimensional Integration of Kinks and Jumps---{{Smoothing}} by Preintegration},
  author  = {Griewank, Andreas and Kuo, Frances Y. and Le{\"o}vey, Hernan and Sloan, Ian H.},
  journal = {Journal of Computational and Applied Mathematics},
  year    = {2018},
  volume  = {344},
  pages   = {259--274},
  issn    = {03770427},
  langid  = {english},
  month   = dec,
  urldate = {2024-03-15}
}

@article{imaiGeneralDimensionReduction2006,
  title     = {A General Dimension Reduction Technique for Derivative Pricing},
  author    = {Imai, Junichi and Tan, Ken Seng},
  journal   = {Journal of Computational Finance},
  year      = {2006},
  volume    = {10},
  number    = {2},
  pages     = {129--155},
  publisher = {RISK PUBLICATIONS},
  urldate   = {2023-12-23}
}

@article{suOptimalImportanceSampling2002,
  title    = {Optimal Importance Sampling in Securities Pricing},
  author   = {Su, Yi and Fu, Michael},
  journal  = {{{Journal}} of {{Computational Finance}}},
  year     = {2002},
  volume   = {5},
  number   = {4},
  pages    = {27--50},
  issn     = {14601559},
  urldate  = {2024-04-22}
}

@article{boyleMonteCarloMethods1997,
  title    = {Monte {{Carlo}} Methods for Security Pricing},
  author   = {Boyle, Phelim and Broadie, Mark and Glasserman, Paul},
  journal  = {Journal of Economic Dynamics and Control},
  year     = {1997},
  volume   = {21},
  number   = {8},
  pages    = {1267--1321},
  issn     = {0165-1889},
  month    = jun,
  series   = {Computational Financial Modelling},
  urldate  = {2023-12-25}
}

@article{kukLaplaceImportanceSampling1999,
  title   = {Laplace {{importance sampling}} for {{generalized linear mixed models}}},
  author  = {Kuk, Anthony Y. C.},
  journal = {Journal of Statistical Computation and Simulation},
  year    = {1999},
  volume  = {63},
  number  = {2},
  pages   = {143--158},
  issn    = {0094-9655, 1563-5163},
  langid  = {english},
  month   = apr,
  urldate = {2024-04-06}
}

@article{liuEstimatingMeanDimensionality2006,
  title    = {Estimating {{mean dimensionality}} of {{analysis}} of {{variance decompositions}}},
  author   = {Liu, Ruixue and Owen, Art B},
  journal  = {Journal of the American Statistical Association},
  year     = {2006},
  volume   = {101},
  number   = {474},
  pages    = {712--721},
  issn     = {0162-1459, 1537-274X},
  langid   = {english},
  month    = jun,
  urldate  = {2023-12-25}
}

@article{boothMaximizingGeneralizedLinear1999,
  title     = {Maximizing {{generalized linear mixed model likelihoods with}} an {{automated Monte Carlo EM algorithm}}},
  author    = {Booth, James G. and Hobert, James P.},
  journal   = {Journal of the Royal Statistical Society Series B: Statistical Methodology},
  year      = {1999},
  volume    = {61},
  number    = {1},
  pages     = {265--285},
  copyright = {https://academic.oup.com/journals/pages/open\_access/funder\_policies/chorus/standard\_publication\_model},
  issn      = {1369-7412, 1467-9868},
  langid    = {english},
  month     = jan,
  urldate   = {2024-04-06}
}

@article{BroadieEstimatingsecurityprice1996,
  title     = {Estimating Security Price Derivatives Using Simulation},
  author    = {Mark Broadie and Paul Glasserman},
  journal   = {Management Science},
  year      = {1996},
  volume    = {42},
  number    = {2},
  pages     = {269--285},
  publisher = {INFORMS},
  issn      = {00251909, 15265501}
}

@article{akessonPathGenerationQuasiMonte2000,
  title     = {Path {{generation}} for {{quasi-Monte Carlo simulation}} of {{mortgage-backed securities}}},
  author    = {{\AA}kesson, Fredrik and Lehoczky, John P.},
  journal   = {Management Science},
  year      = {2000},
  volume    = {46},
  number    = {9},
  pages     = {1171--1187},
  publisher = {INFORMS},
  issn      = {0025-1909},
  month     = sep,
  urldate   = {2023-12-23}
}

@article{glassermanAsymptoticallyOptimalImportance1999,
  title    = {Asymptotically {{optimal importance sampling}} and {{stratification}} for {{pricing path-dependent options}}},
  author   = {Glasserman, Paul and Heidelberger, Philip and Shahabuddin, Perwez},
  journal  = {Mathematical Finance},
  year     = {1999},
  volume   = {9},
  number   = {2},
  pages    = {117--152},
  issn     = {1467-9965},
  langid   = {english},
  urldate  = {2024-04-03}
}

@article{sobolSensitivityEstimatesNonlinear1993,
  title    = {Sensitivity {{estimates}} for {{nonlinear mathematical models}}},
  author   = {Sobol', I M.},
  journal  = {Mathematical Modelling and Computational Experiments},
  year     = {1993},
  volume   = {1},
  number   = {4},
  pages    = {407--414},
  urldate  = {2024-04-11}
}

@article{sobolDerivativeBasedGlobal2009,
  title   = {Derivative based global sensitivity measures and their link with global sensitivity indices},
  author  = {Sobol', I. M. and Kucherenko, S.},
  journal = {Mathematics and Computers in Simulation},
  year    = {2009},
  volume  = {79},
  number  = {10},
  pages   = {3009--3017}
}

@article{kuoQuasiMonteCarloHighly2008,
  title    = {Quasi-{{Monte Carlo}} for {{highly structured generalised response models}}},
  author   = {Kuo, F. Y. and Dunsmuir, W. T. M. and Sloan, I. H. and Wand, M. P. and Womersley, R. S.},
  journal  = {Methodology and Computing in Applied Probability},
  year     = {2008},
  volume   = {10},
  number   = {2},
  pages    = {239--275},
  issn     = {1573-7713},
  langid   = {english},
  month    = jun,
  urldate  = {2024-04-20}
}

@article{wangHandlingDiscontinuitiesFinancial2016,
  title      = {Handling {{discontinuities}} in {{financial engineering}}: {{good path simulation}} and {{smoothing}}},
  author     = {Wang, Xiaoqun},
  journal    = {Operations Research},
  year       = {2016},
  volume     = {64},
  number     = {2},
  pages      = {297--314},
  publisher  = {INFORMS},
  issn       = {0030-364X},
  month      = apr,
  shorttitle = {Handling {{Discontinuities}} in {{Financial Engineering}}},
  urldate    = {2024-04-13}
}

@article{wangEffectsDimensionReduction2006,
  title    = {On the {{effects}} of {{dimension reduction techniques}} on {{some high-dimensional problems}} in {{finance}}},
  author   = {Wang, Xiaoqun},
  journal  = {Operations Research},
  year     = {2006},
  volume   = {54},
  number   = {6},
  pages    = {1063--1078},
  issn     = {0030-364X, 1526-5463},
  langid   = {english},
  month    = dec,
  urldate  = {2023-12-25}
}

@article{wangQuasiMonteCarloMethods2011,
  title      = {Quasi-{{Monte Carlo methods}} in {{financial engineering}}: {{an equivalence principle}} and {{dimension reduction}}},
  author     = {Wang, Xiaoqun and Sloan, Ian H.},
  journal    = {Operations Research},
  year       = {2011},
  volume     = {59},
  number     = {1},
  pages      = {80--95},
  issn       = {0030-364X, 1526-5463},
  langid     = {english},
  month      = feb,
  shorttitle = {Quasi-{{Monte Carlo Methods}} in {{Financial Engineering}}},
  urldate    = {2023-12-23}
}

@article{heErrorRateConditional2019,
  title   = {On the {{error rate}} of {{conditional quasi--Monte Carlo}} for {{discontinuous functions}}},
  author  = {He, Zhijian},
  journal = {SIAM Journal on Numerical Analysis},
  year    = {2019},
  volume  = {57},
  number  = {2},
  pages   = {854--874},
  issn    = {0036-1429, 1095-7170},
  langid  = {english},
  month   = jan,
  urldate = {2024-02-25}
}

@article{liuPreintegrationActiveSubspace2023,
  title    = {Preintegration via {{active subspace}}},
  author   = {Liu, Sifan and Owen, Art B.},
  journal  = {SIAM Journal on Numerical Analysis},
  year     = {2023},
  volume   = {61},
  number   = {2},
  pages    = {495--514},
  issn     = {0036-1429, 1095-7170},
  langid   = {english},
  month    = apr,
  urldate  = {2023-12-17}
}

@article{constantineActiveSubspaceMethods2014,
  title      = {Active {{subspace methods}} in {{theory}} and {{practice}}: {{applications}} to {{Kriging surfaces}}},
  author     = {Constantine, Paul G. and Dow, Eric and Wang, Qiqi},
  journal    = {SIAM Journal on Scientific Computing},
  year       = {2014},
  volume     = {36},
  number     = {4},
  pages      = {A1500-A1524},
  issn       = {1064-8275, 1095-7197},
  langid     = {english},
  month      = jan,
  shorttitle = {Active {{Subspace Methods}} in {{Theory}} and {{Practice}}},
  urldate    = {2023-12-17}
}

@article{wengEfficientComputationOption2017,
  title    = {Efficient {{computation}} of {{option prices}} and {{Greeks}} by {{quasi--Monte Carlo Method}} with {{smoothing}} and {{dimension reduction}}},
  author   = {Weng, Chengfeng and Wang, Xiaoqun and He, Zhijian},
  journal  = {SIAM Journal on Scientific Computing},
  year     = {2017},
  volume   = {39},
  number   = {2},
  pages    = {B298-B322},
  issn     = {1064-8275, 1095-7197},
  langid   = {english},
  month    = jan,
  urldate  = {2023-12-23}
}

@article{zhangEfficientImportanceSampling2021,
  title   = {Efficient {{importance sampling}} in {{quasi-Monte Carlo methods}} for {{computational finance}}},
  author  = {Zhang, Chaojun and Wang, Xiaoqun and He, Zhijian},
  journal = {SIAM Journal on Scientific Computing},
  year    = {2021},
  volume  = {43},
  number  = {1},
  pages   = {B1-B29},
  issn    = {1064-8275, 1095-7197},
  langid  = {english},
  month   = jan,
  urldate = {2023-12-17}
}

@article{liuTransportQuasiMonteCarlo2026,
  title   = {Transport {{quasi-Monte Carlo}}},
  author  = {Liu, Sifan},
  journal = {SIAM Journal on Scientific Computing},
  year    = {2026},
  volume  = {48},
  number  = {4},
  pages   = {C658--C683},
  langid  = {english},
  month   = jul
}

@article{caflischValuationMortgagebackedSecurities1997,
  title    = {Valuation of mortgage-backed securities using {{Brownian}} bridges to reduce effective Dimension},
  author   = {Caflisch, Russel and Morokoff, William and Owen, Art},
  journal  = {The Journal of Computational Finance},
  year     = {1997},
  volume   = {1},
  number   = {1},
  pages    = {27--46},
  issn     = {14601559},
  urldate  = {2023-12-23}
}

@article{paskovFasterValuationFinancial1995,
  title    = {Faster {{valuation}} of {{financial derivatives}}},
  author   = {Paskov, Spassimir H. and Traub, Joseph F.},
  journal  = {The Journal of Portfolio Management},
  year     = {1995},
  volume   = {22},
  number   = {1},
  pages    = {113--120},
  issn     = {0095-4918, 2168-8656},
  langid   = {english},
  month    = oct,
  urldate  = {2023-12-23}
}

@article{Reiderefficientmontecarlo1993,
  title   = {An Efficient {{Monte Carlo}} Technique for Pricing Options},
  author  = {R. Reider},
  journal = {Working Paper, Wharton School, University of Pennsylvania},
  year    = {1993}
}

@inproceedings{acworthComparisonMonteCarlo1998a,
  title     = {A {{comparison}} of {{some Monte Carlo}} and {{quasi Monte Carlo techniques}} for {{option pricing}}},
  author    = {Acworth, Peter A. and Broadie, Mark and Glasserman, Paul},
  booktitle = {Monte {{Carlo}} and {{Quasi-Monte Carlo Methods}} 1996},
  year      = {1998},
  pages     = {1--18},
  publisher = {Springer},
  address   = {New York, NY},
  editor    = {Niederreiter, Harald and Hellekalek, Peter and Larcher, Gerhard and Zinterhof, Peter},
  isbn      = {978-1-4612-1690-2},
  langid    = {english}
}
